\documentclass[11pt]{amsart}
\usepackage{amsmath}
\usepackage{a4wide}
\usepackage{amssymb}
\usepackage{amsopn}
\usepackage{epsfig}
\usepackage{amsfonts}
\usepackage{latexsym}
\usepackage{amsthm}
\usepackage{enumerate}
\usepackage{color}
\usepackage{pgf}
\usepackage{tikz}

\usepackage{mathrsfs}   

\usetikzlibrary{arrows,automata}

\DeclareMathAlphabet{\pazocal}{OMS}{zplm}{m}{n}
\newtheorem{theorem}{Theorem}[section]
\newtheorem{lemma}[theorem]{Lemma}
\newtheorem{proposition}[theorem]{Proposition}
\newtheorem{corollary}[theorem]{Corollary}

\theoremstyle{definition}
\newtheorem{definition}[theorem]{Definition}
\newtheorem{example}[theorem]{Example}

\theoremstyle{remark}
\newtheorem{remark}[theorem]{Remark}

\numberwithin{equation}{section}

\newcommand{\R}{\ensuremath{\mathbb{R}}}
\newcommand{\N}{\ensuremath{\mathbb{N}}}
\newcommand{\Z}{\ensuremath{\mathbb{Z}}}
\newcommand{\TT}{\mathbb{T}}

\renewcommand{\S} {\ensuremath{\widetilde{\mathbf{B}}}}

\newcommand{\wS}{\widetilde{S}}
\newcommand{\wL}{\widetilde{L}}
\newcommand{\wE}{\widetilde{E}}

\newcommand{\wB}{\widetilde{B}}
\newcommand{\cR}{\mathcal{R}}

\newcommand{\sa} { {\mathbf{a}}}
\renewcommand{\sb}{\mathbf{b}}
\renewcommand{\ss}{\mathbf{s}}
\newcommand{\set}[1]{\left\{#1\right\}}

\newcommand{\ga}{\gamma}
\newcommand{\Ga}{\Gamma}
\newcommand{\tGa}{\widetilde{\Gamma}}
\newcommand{\ep}{\varepsilon}
\newcommand{\f}{\infty}

\newcommand{\al}{\alpha}

\newcommand{\CC}{\mathscr{C}}
\newcommand{\lle}{\preccurlyeq}
\newcommand{\lge}{\succcurlyeq}

\renewcommand{\b}{\mathbf{b}}
\newcommand{\si}{\sigma}
\newcommand{\La}{\Lambda}

\begin{document}
	
\title[]{Density spectrum of Cantor measure}

\author[Pieter Allaart]{Pieter Allaart$^*$}
\address[P. Allaart]{Mathematics Department, University of North Texas, 1155 Union Cir \#311430, Denton, TX 76203-5017, U.S.A.}
\email{allaart@unt.edu}
	
	\author[Derong Kong]{Derong Kong}
	\address[D. Kong]{College of Mathematics and Statistics, Chongqing University,  401331, Chongqing, P.R.China}
	\email{derongkong@126.com}

	\dedicatory{}



\begin{abstract}
Given $\rho\in(0, 1/3]$, let $\mu$ be the Cantor measure satisfying $\mu=\frac{1}{2}\mu  f_0^{-1}+\frac{1}{2}\mu  f_1^{-1}$, where $f_i(x)=\rho x+i(1-\rho)$ for $i=0, 1$. The support of $\mu$ is a Cantor set $C$ generated by the iterated function system $\{f_0, f_1\}$. Continuing the work of Feng et al.~(2000) on the pointwise lower and upper densities
\[
\Theta_*^s(\mu, x)=\liminf_{r\to 0}\frac{\mu(B(x,r))}{(2r)^s},\qquad \Theta^{*s}(\mu, x)=\limsup_{r\to 0}\frac{\mu(B(x,r))}{(2r)^s},
\]
where $s=-\log 2/\log\rho$ is the Hausdorff dimension of $C$, we give a complete description of the sets $D_*$ and $D^*$  consisting of all possible values of the lower and upper densities, respectively. We show that both sets contain infinitely many isolated and infinitely many accumulation points, and they have the same Hausdorff dimension as the Cantor set $C$. Furthermore, we compute the Hausdorff dimension of the level sets of the lower and upper densities.
{Our method consists in formulating an equivalent ``dyadic" version of the problem involving the doubling map on $[0,1)$, which we solve by using known results on the entropy of a certain open dynamical system and the notion of tuning.}
\end{abstract}
	
\thanks{$^*$ Corresponding author.}
	
	\maketitle

\section{Introduction} \label{sec:introduction}

For $0<\rho<1/2$ let $C=C_\rho$ be the Cantor set generated by the iterated function system
\[
f_0(x)=\rho x,\qquad f_1(x)=\rho x+(1-\rho).
\]
Then
\begin{equation} \label{eq:Cantor-set-dimension}
\dim_H C=\frac{\log 2}{-\log\rho}=:s(\rho)=s,
\end{equation}
where $\dim_H$ denotes Hausdorff dimension. Let $\mu=\mu_\rho$ be the natural probability measure on $C$, that is, the unique probability measure such that
\[
\mu(A)=\frac{1}{2}\mu  \big(f_0^{-1}(A)\big)+\frac{1}{2}\mu \big( f_1^{-1}(A)\big)\quad\textrm{for any Borel set }A\subseteq \R.
\]
The measure $\mu$ is called a \emph{Cantor measure}, and is equal to the normalized $s$-dimensional Hausdorff measure restricted to $C$ (cf.~\cite{Falconer_1990}).  Given a point $x\in C$, the pointwise \emph{lower and upper $s$-densities} of $\mu$ at $x$ are defined respectively by
\[
\Theta_*^s(\mu, x):=\liminf_{r\to 0}\frac{\mu(B(x, r))}{(2r)^s}\qquad\mbox{and}\qquad \Theta^{*s}(\mu,x):=\limsup_{r\to 0}\frac{\mu(B(x, r))}{(2r)^s},
\]
where $B(x, r):=(x-r, x+r)$.
Densities of Cantor measures have been studied extensively (cf.~\cite{Falconer_1990, Mattila_1995}). The following two properties of $\mu$ are well known:
\begin{itemize}
  \item There exist finite positive constants $a$ and $b$ such that $a r^s\le \mu(B(x, r))\le b r^s$ for any $x\in C$ and $0<r\le 1$.
  \item There exist $0<d_*<d^*<\f$ such that for $\mu$-almost every $x\in C$ we have $\Theta_*^s(\mu, x)=d_*$ and $\Theta^{*s}(\mu, x)=d^*$.
\end{itemize}
We observe that there exist alternative density constructions for measures which may be better suited for the study of  Cantor measures; for instance the C\'esaro averaging construction in \cite{BF-1992} yields an almost everywhere density for the Cantor measure $\mu$.

In 2000, Feng, Hua and Wen \cite{Feng-Hua-Wen-2000} explicitly calculated the pointwise densities $\Theta_*^s(\mu, x)$ and $\Theta^{*s}(\mu, x)$ for every $x\in C$. To summarize their main results we first introduce a $2$-to-$1$ map $T: [0, \rho]\cup [1-\rho, 1]~\to [0, 1]$ defined by
\begin{equation*} 
T(x):=\begin{cases}
  \frac{x}{\rho} & \textrm{if}\ 0\le x\le \rho,\\
  \frac{x-(1-\rho)}{\rho} & \textrm{if}\ 1-\rho\le x\le 1.
\end{cases}
\end{equation*}
{Thus, the inverse branches of $T$ are the maps $f_0$ and $f_1$.}
Then for each $x\in C$ we set
\[
\tau(x):=\min\set{\liminf_{n\to\f}T^n (x), \liminf_{n\to\f}T^n(1-x)}.
\]

\begin{theorem}[Feng et al.~\cite{Feng-Hua-Wen-2000}] \label{th:feng-etal}
Let $0<\rho\le 1/3$. Then for any $x\in C$ the pointwise lower $s$-density of $\mu$ at $x$ is given by
\[
\Theta_*^s(\mu, x)=\frac{1}{2^{1+s}(1-\rho-\tau(x))^s},
\]
and the pointwise upper $s$-density of $\mu$ at $x$ is given by
\[
\Theta^{*s}(\mu, x)=\begin{cases}
  2^{-s} & \textrm{if}\ T^n(x)\in\set{0, 1}\textrm{ for some }n\ge 0,\\
  \frac{1}{2^s(1-\rho+\rho\tau(x))^s}& \textrm{otherwise}.
\end{cases}
\]
\end{theorem}

Observe that $\tau(x)=0$ for $\mu$-almost every $x\in C$, so by Theorem \ref{th:feng-etal} the almost sure lower and upper $s$-densities of $\mu$ are
\[
d_*=\frac{1}{2^{1+s}(1-\rho)^s}\quad\textrm{and}\quad d^*=\frac{1}{2^s(1-\rho)^s}.
\]
In particular, $d^*=2d_*$, independent of $\rho$.
Wang, Wu and Xiong \cite{Wang-Wu-Xiong-2011} later extended Theorem \ref{th:feng-etal} to the interval $0<\rho\leq\rho^*$, where $\rho^*\approx .351811$ satisfies $(1-\rho^*)^{s(\rho^*)}=3/4$. However, they showed that for $\rho>\rho^*$, the formula for the upper density changes. Recently, Kong, Li and Yao \cite{KLY-2021} extended Theorem \ref{th:feng-etal} to the case of a general homogeneous Cantor measure.

One might ask what the range of possible values is for $\Theta_*^s(\mu,x)$ and $\Theta^{*s}(\mu,x)$. Noting that both of these numbers depend on $x$ only through the quantity $\tau(x)$, it is sufficient to determine the set
\[
\Ga:=\set{\tau(x): x\in C},
\]
which we call the \emph{density spectrum} of the Cantor measure $\mu$. Furthermore, we can decompose the Cantor set $C$ according to the values of the lower (or upper) density, and this reveals an interesting multifractal structure. Specifically, we are interested in the Hausdorff dimension of the level sets
\begin{equation*} 
L_*(y):=\set{x\in C: \Theta_*^s(\mu, x)=y}\quad\textrm{and}\quad L^*(z):=\set{x\in C:\Theta^{*s}(\mu, x)=z}
\end{equation*}
for $y\in D_*:=\set{\Theta_*^s(\mu, x): x\in C}$  and $z\in D^*:=\set{\Theta^{*s}(\mu, x): x\in C}$. It is sufficient to study the level sets
\[
L(t):=\{x\in C: \tau(x)=t\}, \qquad t\in\Ga,
\]
in view of the identities
\[
L_*(y)=L\left(1-\rho-\frac12(2y)^{-1/s}\right), \qquad L^*(z)=L\left(\frac{2(1-\rho)-z^{-1/s}}{2\rho}\right).
\]

Our results presented below are valid for all $0<\rho<1/2$. However, the reader should bear in mind that they can be interpreted in terms of the upper and lower densities of $\mu$ only when $\rho\leq\rho^*\approx .351811$.

A first observation is that $\Ga\subseteq C$, since $\liminf_{n\to\f}T^n (x)\in C$ and $\liminf_{n\to\f}T^n(1-x)\in C$ for every $x\in C$. In \cite[Theorem 1.3]{Wang-Wu-Xiong-2011} Wang, Wu and Xiong initiated the study of $\Ga$, but their characterization was not fully explicit. By connecting the quantities $\tau(x)$ with the entropy of a certain open dynamical system, we are able to give a much more explicit description of $\Gamma$. Our first main result concerns the topological properties of $\Ga$.

\begin{theorem} \label{main:Ga-topology}
The density spectrum $\Ga$ can be represented as
\begin{equation} \label{eq:Gamma-representation}
\Ga=\set{t\in C: t\le T^n(t)\le 1-t~\forall n\ge 0}.
\end{equation}
Therefore, $\Ga$ is a closed subset of $C$, and it contains both infinitely many isolated and infinitely many accumulation points. Furthermore,
  \begin{enumerate}[{\rm(i)}]
  \item $t$ is isolated in $\Ga$ if and only if $T^n (t)=1-t$ for some $n\in\N$;
  \item the smallest element $0$ of $\Ga$ is an accumulation point, while
  the largest element ${\rho}/({1+\rho})$ of $\Ga$ is an isolated point.
\end{enumerate}
\end{theorem}

\begin{remark}
The representation \eqref{eq:Gamma-representation} shows that $\Ga$ is closely related to the set of kneading invariants for unimodal maps introduced by Allouche and Cosnard \cite{Allouche_Cosnard_1983,Allouche-Cosnard-2001}, namely
\begin{equation} \label{eq:Gamma-AC-def}
\Ga_{AC}:=\{x\in[0,1): D^n(x)\in[1-x,x]\ \forall n\geq 0\},
\end{equation}
where $D:[0,1)\to[0,1)$ is the doubling map $D(x):=2x \pmod 1$. We will explain this connection, and use it to prove Theorem \ref{main:Ga-topology}, in Section \ref{sec:Gamma}.
\end{remark}

Our second main result describes the dimensional properties of $\Ga$. First we introduce a family of auxiliary sets
\begin{equation} \label{eq:S_tau}
S(t):=\set{x\in C: \tau(x)\ge t},\quad t\in[0,1].
\end{equation}

\begin{theorem} \label{main:Ga-dimension}
\begin{enumerate}[{\rm(i)}]
\item The set $\Ga$ has full Hausdorff dimension:
  \[
	\dim_H\Ga=\dim_H C=s(\rho).
	\]
 \item
 For any $t\in(0, 1]$ we have
 \[
  \dim_H(\Ga\cap[t, 1])=\dim_H S(t)<\dim_H\Ga.
 \]
 \item The map
	\begin{equation} \label{eq:psi}
	\delta: t\mapsto\dim_H(\Ga\cap[t, 1])
	\end{equation}
	is a non-increasing devil's staircase on $[0, 1]$, i.e., $\delta$ is non-increasing, continuous and locally constant almost everywhere in $[0, 1]$; see Figure \ref{fig:1}.
\end{enumerate}
\end{theorem}

It follows from Theorem \ref{main:Ga-dimension} (ii) that $\dim_H(\Ga\cap[0, t])=\dim_H\Ga$ for all $t>0$. In that sense, we could say that the dimension of $\Gamma$ is concentrated near $0$.

The graph of $\delta$ is shown in Figure \ref{fig:1}. It first becomes zero at $t=t_F:=1-\pi(\tau_1\tau_2\dots)$, where $(\tau_i)_{i=0}^\f$ is the Thue-Morse sequence (see Section \ref{sec:entropy}) and $\pi:\{0,1\}^\N\to C$ is the natural projection (see \eqref{eq:pi} below).

\begin{center}
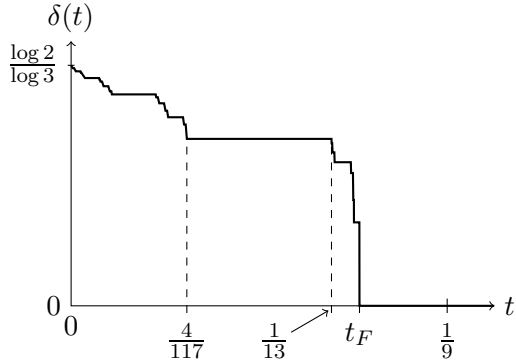
\begin{figure}[h!]
\begin{tikzpicture}[xscale=4.5,yscale=3.2]
\draw[<->] (0,1.1)--(0,0) node[anchor=east]{$0$} node[anchor=north]{$0$} -- (1.25,0) node[anchor=west]{$t$};
\draw (0.852,-0.03) node[anchor=north]{$t_F$} --(0.852,0.03);
\draw (1.111,-0.03) node[anchor=north]{$\frac19$} --(1.111,0.03);
\draw (-0.01,1) node[anchor=east]{$\frac{\log 2}{\log 3}$} -- (0.01,1);
\draw(0,1.3) node[anchor=north]{$\delta(t)$};
\draw[thick] (0,1)--(.0046,.988)--(.0091,.988)
--(.0137,.975)--(.0275,.975)
--(.0290,.969)--(.0305,.969)
--(.0408,.947)--(.0826,.947)
--(.0838,.940)--(.0854,.940)
--(.0868,.932)--(.0915,.932)
--(.0959,.913)--(.1099,.913)
--(.1143,.892)--(.1189,.892)
--(.1204,.879)--(.2500,.879)
--(.2514,.867)--(.2562,.867)
--(.2603,.842)--(.2747,.842)
--(.2788,.811)--(.2836,.811)
--(.2867,.784)--(.3306,.784)
--(.3336,.754)--(.3385,.754)
--(.3419,.694)--(.7692,.694)
--(.7696,.676)--(.7713,.676)
--(.7724,.638)--(.7777,.638)
--(.7785,.597)--(.8264,.597)
--(.8272,.552)--(.8326,.552)
--(.8336,.440)--(.8354,.440)
--(.8356,.347)--(.8516,.347)
--(.8519,0)--(1.25,0);
\draw[dashed] (40/117,-0.03) node[anchor=north]{$\tfrac{4}{117}$} --(40/117,.694);
\draw[dashed] (10/13,-0.03) -- (10/13,.694);
\draw[->] (.65,-.11)--(.76,-.02);
\draw(.6,-.145) node{$\tfrac{1}{13}$};
\end{tikzpicture}
\caption{The graph of $\delta: t\mapsto\dim_H(\Ga\cap[t,1])=\dim_H {S(t)}$ for $\rho=1/3$. The dimension becomes zero at $t_F\approx 0.08519$ (see Remark \ref{rem:devil-staircase-detail} and Example \ref{ex:longest-plateau} for more details).}
\label{fig:1}
\end{figure}
\end{center}

\begin{remark}
The equality in Theorem \ref{main:Ga-dimension} (ii) is an instance of the interplay between the ``parameter space" (in this case, $\Ga$) and the ``dynamical space" (in our case $S(t)$) which was first observed by Douady \cite{Douady-1995} in the context of dynamics of real quadratic polynomials. A similar result was proved by Tiozzo \cite{Tiozzo_2015}, who considers for $c\in\R$ the set of angles of external rays which ``land" on the real slice of the Mandelbrot set to the right of $c$ (parameter space) and the set of external angles which land on the real slice of the Julia set of the map $z\mapsto z^2+c$ (dynamical space), showing that these two sets have the same Hausdorff dimension. In fact it is possible to deduce Theorem \ref{main:Ga-dimension} from Tiozzo's main result; however, we choose a slightly different approach, which will aid us in also computing the dimension of the level sets $L(t)$.

More recently, Carminati and Tiozzo \cite{CT-2021} proved an analogous identity in the context of continued fractions. While all of these results are different, the proofs are based on very similar ideas.
\end{remark}

From the definition of $\tau$ it is clear that the level set $L(t)$ is dense in $C$ for any $t\in\Ga$. Our last main result {concerns} the Hausdorff dimension of $L(t)$. In order to state it, we need the \emph{bifurcation set} of the map $\delta$ from \eqref{eq:psi}, defined by
\begin{equation} \label{eq:bifurcation-t}
E:=\set{t: \delta(t+\ep)<\delta(t-\ep)\ \forall\, \ep>0}.
\end{equation}
We also let $\Ga_{iso}$ denote the set of isolated points of $\Ga$. 

\begin{theorem} \label{main:level-set}
For any $t\in\Ga$ we have
\begin{equation} \label{eq:level-set-and-local-dimension}
\dim_H L(t)=\lim_{\ep\to 0}\dim_H\big(\Ga\cap(t-\ep, t+\ep)\big).
\end{equation}
In other words, the Hausdorff dimension of $L(t)$ is equal to the local dimension of $\Ga$ at $t$, for every $t$.
Furthermore,
  \begin{enumerate}[{\rm(i)}]

  \item if $t\in\Ga_{iso}$, then
  $L(t)$ is countably infinite;

	\item the bifurcation set $E\subset \Ga$ is a Cantor set containing $0$, and $\dim_H E=\dim_H\Ga$;

  \item if $t\in E$, then
  \begin{equation} \label{eq:level-S}
  \dim_H L(t)=\dim_H (\Ga\cap[t, 1])=\dim_H S(t).
  \end{equation}
  \end{enumerate}
\end{theorem}

\begin{remark}\mbox{}
\begin{enumerate}
  [{\rm(i)}]
  \item Since $\Ga$ has zero Lebesgue measure, the local dimension $\lim_{\ep\to 0}\dim_H(\Gamma\cap(t-\ep, t+\ep))$ is a reasonable quantity to describe the distribution of the set $\Ga$. For other examples of the local dimension of Lebesgue null sets, see \cite{Kalle-Kong-Li-Lv-2016, Allaart-Kong-2018} and the references therein.

  \item We can actually compute the dimension of $L(t)$ for all $t\in\Ga$. However, this requires more notation and concepts, and will be postponed until {Remark \ref{rem:renormalized-level-set}}.

  \item In contrast with Theorem \ref{main:Ga-dimension}, Theorem \ref{main:level-set} does not, to the best of our knowledge, have an analogue in the theory of dynamics of real quadratic polynomials. In order to prove Theorem \ref{main:level-set}, we borrow a technique from the literature on unique expansions in non-integer bases {(see \cite{Allaart-Kong-2018})}.
\end{enumerate}
\end{remark}

Observe that the maps $\tau\mapsto (1-\rho-\tau)^{-s}$ and $\tau\mapsto (1-\rho+\rho\tau)^{-s}$ are both bi-Lipschitz on $\Ga$. Note also that the first map is increasing and the second map is decreasing. Thus, using Theorem \ref{th:feng-etal}
it is easy to translate the above results into analogous statements concerning the sets $D_*$ and $D^*$, and the level sets $L_*(y)$ and $L^*(z)$. We leave this as an exercise for the interested reader.

\medskip

The rest of this article is organized as follows. In Section \ref{sec:preliminaries} we develop notation and translate the problem to an equivalent, but easier to analyze, dyadic version involving the doubling map on the unit interval. In particular we introduce the map $\psi$ which conjugates the map $T$ with the doubling map $D$. We then state analogues of our main theorems in the dyadic setting.
In Section \ref{sec:Gamma} we use known results on the set $\Ga_{AC}$ of Allouche and Cosnard to prove the dyadic analogue of Theorem \ref{main:Ga-topology}.

In Section \ref{sec:entropy} we consider an open dynamical system involving the doubling map, whose survivor set $K(\theta)$ turns out to be closely related to the set $S(t)$. We review known properties of $K(\theta)$, including its Hausdorff dimension and the corresponding bifurcation set $\cR$. We then introduce the tuning maps $\Psi_I$ and show that Hausdorff dimension behaves nicely under tuning.

In Section \ref{sec:R-local-dimension} we state and prove an analogue of Theorem \ref{main:Ga-dimension} (ii) for the bifurcation set $\cR$, and use it to compute the local dimension of $\cR$. This last result appears to be new and may be of independent interest, since the set $\cR$ has occurred several times before in the dynamical systems literature and has a variety of interpretations (see, e.g. \cite{Bon-Car-Ste-Giu-2013,Tiozzo_2015}).

In Section \ref{sec:Gamma-dimension} we prove the dyadic analogue of Theorem \ref{main:Ga-dimension} by introducing a family of sets $\wB(t)$ as a bridge between the dyadic version $\wS(t)$ of $S(t)$ and the survivor set $K(\theta)$, and using the results from Sections \ref{sec:entropy} and \ref{sec:R-local-dimension}.

The longest and most technical section of the paper is Section \ref{sec:level-set}, where we prove the dyadic analogue of Theorem \ref{main:level-set} and give a more complete description of the level sets $\wL(t)$ of the dyadic analogue of $\tau(x)$, using the tuning map from Section \ref{sec:entropy}.

Finally, in Section \ref{sec:main-proofs} we quickly derive our main results (Theorems \ref{main:Ga-topology}, \ref{main:Ga-dimension} and \ref{main:level-set}) from their dyadic counterparts by using bi-H\"older properties of the map $\psi$ which conjugates $T$ with the doubling map $D$. We end the paper with a concrete example.

\section{Preliminaries} \label{sec:preliminaries}

We recall some terminology from symbolic dynamics (cf.~\cite{Lind_Marcus_1995}). Let $\Omega:=\set{0, 1}^\N$ be the set of all infinite sequences of zeros and ones. Then $(\Omega, \si)$ is a shift space, where $\si$ is the left shift defined by $\si((d_i))=(d_{i+1})$ for any $(d_i)\in\Omega$. By a \emph{word} we mean a finite string of zeros and ones. Denote by $\set{0,1}^*$ the set of all finite words including the empty word $\epsilon$. For two words $\mathbf c=c_1\ldots c_m$ and $\mathbf d=d_1\ldots d_n$ we write $\mathbf{cd}=c_1\ldots c_md_1\ldots d_n$ for their concatenation. In particular, for any $k\in\N$ we denote by $\mathbf c^k$ the $k$-fold concatenation of $\mathbf c$ with itself, and by $\mathbf c^\f$ the periodic sequence which is obtained by the infinite concatenation of $\mathbf c$ with itself. For a word $\mathbf c=c_1\ldots c_m$, we denote its \emph{length} by $|\mathbf c|=m$. If $c_m=0$ we write $\mathbf c^+:=c_1\ldots c_{m-1}1$; and if $c_m=1$, we write $\mathbf c^-:=c_1\ldots c_{m-1}0$. Furthermore, we denote by $\overline{\mathbf c}:=(1-c_1)\ldots (1-c_m)$ the \emph{reflection} of $\mathbf c$; and for an infinite sequence $(c_i)\in\Omega$ we denote its reflection by $\overline{(c_i)}:=(1-c_1)(1-c_2)\ldots$. Throughout the paper we will use the lexicographical order `$\prec, \lle, \succ$' or `$\lge$' between sequences and words. For example, for two sequences $(c_i), (d_i)\in\Omega$ we say $(c_i)\prec (d_i)$ if $c_1<d_1$, or there exists $n\in \N$ such that $c_1\ldots c_n=d_1\ldots d_n$ and $c_{n+1}<d_{n+1}$. Similarly, for two words $\mathbf c, \mathbf d\in\set{0,1}^*$ we write $\mathbf c\prec \mathbf d$ if $\mathbf c 0^\f\prec \mathbf d 0^\f$.

Note that for each $x\in C$ there exists a sequence $(d_i)=d_1d_2\ldots\in\Omega$, called the {\em coding} of $x$,  such that
\begin{equation} \label{eq:pi}
x=\pi((d_i)):=(1-\rho)\sum_{i=1}^\f d_i\rho^{i-1}.
\end{equation}
Since $\rho\in(0, 1/3]$, the coding map $\pi: \Omega \to C$ is bijective. Thus, each $x\in C$ has a unique coding.

Let $\ga$ be the metric on $\Omega$ defined by
\begin{equation} \label{eq:metric}
\ga((c_i), (d_i))=2^{-\inf\set{i\ge 1: c_i\ne d_i}}.
\end{equation}
Then the topology induced by $\ga$ coincides with the order topology on $\Omega$.
Equipped with this metric $\ga$ we can define the Hausdorff dimension $\dim_H$ for any subset of $\Omega$. The following lemma is immediate:

\begin{lemma} \label{lem:bi-Holder}
The projection map $\pi:\Omega \to C$ defined by \eqref{eq:pi} is bi-H\"older continuous with exponent $1/s(\rho)$, where $s(\rho)$ was defined in \eqref{eq:Cantor-set-dimension}. 
\end{lemma}

It is convenient to transfer the problem from the Cantor set $C$ to the unit interval $[0,1)$, or more precisely, the torus $\mathbb{T}:=\R/\Z$. Let $D:\TT\to\TT$ be the {\em doubling map}, defined by
\[
D(x):=2x \!\!\!\mod 1.
\]
The analogue of the function $\tau$ for $D$ is the map
\[
\tilde{\tau}(x):=\min\left\{\liminf_{n\to\infty}D^n(x),\liminf_{n\to\infty}D^n(1-x)\right\}, \qquad x\in [0,1).
\]
We now set
\[
\widetilde{\Gamma}:=\{\tilde{\tau}(x): x\in [0,1)\}.
\]
A first observation is that $\tGa$ does not contain any dyadic rational numbers other than $0$. This is because, if $D^n(x)$ is very close to such a dyadic rational, then $D^{n+m}(x)$ will be very close to $0$ or $1$ for some $m\geq 1$. Hence, every point $x\in\tGa$ has a unique binary expansion
\[
b(x)=b_1 b_2\dots\in\Omega\quad\textrm{such that}\quad x=\sum_{i=1}^\infty 2^{-i}b_i.
\]
{Let $\TT':=\TT\backslash\{k2^{-n}: k\in\N, n\in\N\}$ (here we adopt the convention $\N=\{1,2,3,\dots\}$)}, so $\tGa\subseteq \TT'$. We let $\pi_2:\Omega\to[0,1]$ denote the projection map
\[
\pi_2((d_i))=\sum_{i=1}^\infty \frac{d_i}{2^i}.
\]
Note that $\pi_2\circ b(x)=x$ for all $x\in\TT'$, and $\pi_2\circ \si=D\circ\pi_2$.
Define the map
\[
\psi: C\to [0,1]; \qquad \psi(x):=\pi_2\circ \pi^{-1}(x).
\]
Thus $\psi(x)$ is the unique number in $[0,1]$ whose binary expansion equals the coding of $x$. (For definiteness, we set $\psi(1):=1$.)
Let $C':=\psi^{-1}(\TT')$; then $C'$ is the Cantor set $C$ with all endpoints of ``basic intervals" (except 0) removed. Note that $\psi$ induces an increasing homeomorphism between $C'$ and $\TT'$.
Observe now that $D^n\circ\psi(x)=\psi\circ T^n(x)$ for all $n\geq 0$ and $x\in C'$, and since $\psi$ is increasing, it follows that
\begin{equation} \label{tau-and-tau-tilde}
\tau(x)=\psi^{-1}\circ\tilde{\tau}\circ\psi(x), \qquad x\in C'.
\end{equation}
Thus, $\tGa=\psi(\Ga)$ (see Figure \ref{figure:2}).

We can now also define the analogues of the sets $S(t)$ and $L(t)$, namely
\begin{equation*} 
\wS(t):=\{x\in[0,1): \tilde{\tau}(x)\geq t\}, \qquad \wL(t):=\{x\in[0,1): \tilde{\tau}(x)=t\}.
\end{equation*}
Then for all $t\in C'\backslash\{0\}$ we have by \eqref{tau-and-tau-tilde},
\begin{equation} \label{eq:psi-correspondence}
S(t)=\psi^{-1}\big(\wS(\psi(t))\big), \qquad L(t)=\psi^{-1}\big(\wL(\psi(t))\big).
\end{equation}
(The first equality does not hold for $t=0$, but evidently $S(0)=C$.) We will show in Section \ref{sec:main-proofs} that the function $\psi$, suitably restricted, is bi-H\"older continuous with exponent $s(\rho)$. This allows us to deduce dimensional results for $S(t), L(t)$ and $\Ga$ from the corresponding statements about $\wS(t), \wL(t)$ and $\tGa$; see Section \ref{sec:main-proofs}.

\begin{figure}
  \centering
\usetikzlibrary{matrix}

\begin{tikzpicture}[scale=5cm]
  \matrix (m) [matrix of math nodes,row sep=4.5em,column sep=5em,minimum width=4em]
  {
    C' & C' \\
     \TT' & \TT' \\};
  \path[-stealth]
    (m-1-1) edge node [left] {$\psi$} (m-2-1)
            edge   node [above] {$T$} (m-1-2)
    (m-2-1.east|-m-2-2) edge  node [above] {$D$} (m-2-2)
    (m-1-2) edge node [right] {$\psi$} (m-2-2);
\end{tikzpicture}
\qquad
\begin{tikzpicture}[scale=5cm]
  \matrix (m) [matrix of math nodes,row sep=4em,column sep=1em,minimum width=4em]
  {
    \Gamma & S(t) & L(t) \\
     \tilde\Gamma & \wS(\psi(t))&\wL(\psi(t))\\};
  \path[-stealth]
    (m-1-1) edge node [right] {$\psi$} (m-2-1)
    (m-1-2) edge node [right] {$\psi$} (m-2-2)
    (m-1-3) edge node [right] {$\psi$} (m-2-3);
\end{tikzpicture}
\caption{The isomorphism $\psi$ between the two dynamical systems $(\TT', D)$ and $(C', T)$; and the correspondence between the sets $\tilde\Gamma, \tilde S(\psi(t)), \tilde L(\psi(t))$ in $\TT'$ and the sets $\Gamma, S(t), L(t)$ in $C'$.}\label{figure:2}
\end{figure}
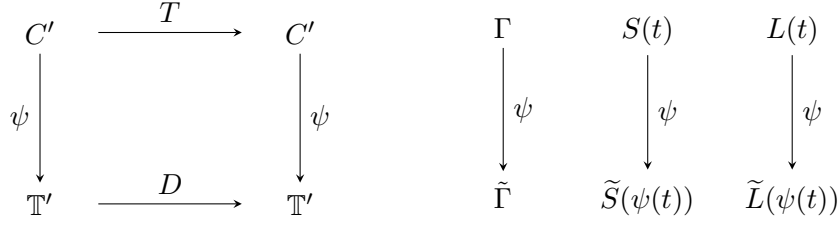

We now state analogs of our main results for the doubling map.

\begin{theorem} \label{main2:Ga-topology}
The density spectrum $\tGa$ can be represented as
\begin{equation} \label{eq:Gamma-representation2}
\tGa=\set{t\in [0,1]: t\le D^n(t)\le 1-t~\forall n\ge 0}.
\end{equation}
Therefore, $\tGa$ is closed, and it contains both infinitely many isolated and infinitely many accumulation points. Furthermore,
  \begin{enumerate}[{\rm(i)}]
  \item $t$ is isolated in $\tGa$ if and only if $D^n (t)=1-t$ for some $n\in\N$;
  \item the smallest element $0$ of $\tGa$ is an accumulation point, while
  the largest element $1/3$ of $\tGa$ is an isolated point.
\end{enumerate}
\end{theorem}

\begin{theorem} \label{main2:Ga-dimension}
\begin{enumerate}[{\rm(i)}]
\item The set $\tGa$ has full Hausdorff dimension: $\dim_H \tGa=1$.
\item
 For any $t\in(0, 1]$ we have
 \[
  \dim_H(\tGa\cap[t, 1])=\dim_H \wS(t)<1.
 \]
\item The map
	\begin{equation} \label{eq:psi2}
	\tilde{\delta}: t\mapsto\dim_H(\tGa\cap[t, 1])
	\end{equation}
	is a non-increasing devil's staircase on $[0, 1]$, i.e., {$\tilde{\delta}$} is non-increasing, continuous and locally constant almost everywhere in $[0, 1]$.
\end{enumerate}
\end{theorem}

For the last theorem, we define the ``dyadic" analog of the bifurcation set $E$ by
\begin{equation} \label{eq:bifurcation-dyadic}
\wE:=\set{t: \tilde{\delta}(t+\ep)<\tilde{\delta}(t-\ep)\ \forall\, \ep>0}.
\end{equation}

\begin{theorem} \label{main2:level-set}
For any $t\in\tGa$ we have
\begin{equation} \label{eq:level-set-and-local-dimension-2}
\dim_H \wL(t)=\lim_{\ep\to 0}\dim_H\big(\tGa\cap(t-\ep, t+\ep)\big).
\end{equation}
Furthermore,
  \begin{enumerate}[{\rm(i)}]

  \item if $t$ is an isolated point of $\tGa$, then $\wL(t)$ is countably infinite;

	\item the bifurcation set $\wE\subset \tGa$ is a Cantor set containing $0$, and $\dim_H \wE=1$;

  \item if $t\in \widetilde{E}$, then $\dim_H \wL(t)=\dim_H \wS(t)$.
  \end{enumerate}
\end{theorem}

The following notation and lemma are important for relating the Hausdorff dimensions of various sets.
For $k\in\N$, let $\mathbf{X}_k$ be the set of sequences in $\Omega$ that do not contain the word $10^k$ or $01^k$.

\begin{lemma} \label{lem:bi-Lip}
For each $k\in\N$, the restriction $\pi_2|_{\mathbf{X}_k}$ is bi-Lipschitz continuous with respect to the metric $\gamma$ from \eqref{eq:metric}.
\end{lemma}

\begin{proof}
Clearly $\pi_2$ is Lipschitz on $\Omega$, so we only need to verify that the inverse of $\pi_2|_{\mathbf{X}_k}$ satisfies a Lipschitz condition. Take $(c_i),(d_i)\in \mathbf{X}_k$, and suppose $\gamma((c_i),(d_i))=2^{-m}$, so $c_i=d_i$ for $1\leq i<m$, and $c_m\neq d_m$. Without loss of generality assume $c_m=0$ and $d_m=1$. Let $n_0:=\min\{n>m:{c_n}=0\}$. Then $n_0\leq m+k$.
 This implies
\[
\pi_2((d_i))-\pi_2((c_i))\geq 2^{-n_0}\geq 2^{-(m+k)}=2^{-k}\gamma((c_i),(d_i)),
\]
giving the desired result.
\end{proof}

At various places in the paper, we use the notion of admissible word, defined below.

\begin{definition} \label{def:admissible}
A word $\sa=a_1\dots a_m\in\{0,1\}^m$ with $m\geq 2$ is {\em admissible} if
\begin{equation} \label{eq:admissible-definition}
\overline{a_1\ldots a_{m-i}}\lle a_{i+1}\ldots a_m\prec a_1\ldots a_{m-i}\quad\forall\, 1\le i<m.
\end{equation}
\end{definition}

We observe that an admissible word must necessarily begin with a 1 and end with a 0.
Admissible words were introduced by Komornik and Loreti \cite{Komornik_Loreti_2007}, who used them to characterize the unique beta expansions of $1$. But they are also useful for a symbolic description of the set $\Ga_{AC}$.

The following elementary lemma will be needed in Section \ref{sec:level-set}.

\begin{lemma} \label{lem:admissible-word-property}
Let $\sa=a_1\dots a_m$ be an admissible word which is not of the form $\sb\overline{\sb}$. Then
\[
\overline{a_{i+1}\dots a_m a_1\dots a_i}\prec \sa \qquad \forall\,0\leq i<m.
\]
\end{lemma}

\begin{proof}
Two applications of \eqref{eq:admissible-definition} (the second with $m-i$ in place of $i$) yield $\overline{a_{i+1}\dots a_m}\lle a_1\dots a_{m-i}$ and $\overline{a_1\dots a_i}\lle a_{m-i+1}\dots a_m$, and hence $\overline{a_{i+1}\dots a_m a_1\dots a_i}\lle \sa$. Suppose we have equality:
\begin{equation} \label{eq:cut-and-reflect}
\overline{a_{i+1}\dots a_m a_1\dots a_i}=a_1\dots a_m.
\end{equation}
If $i<m-i$, then applying \eqref{eq:cut-and-reflect} twice gives $\overline{a_{2i+1}\dots a_m}=a_{i+1}\dots a_{m-i}=\overline{a_1\dots a_{m-2i}}$, so that $a_{2i+1}\dots a_m=a_1\dots a_{m-2i}$, contradicting \eqref{eq:admissible-definition}. Hence $i\geq m-i$. On the other hand, by swapping out the first $m-i$ digits and the last $i$ digits on both sides of \eqref{eq:cut-and-reflect} and taking reflections we obtain
\[
\overline{a_{m-i+1}\dots a_m a_1\dots a_{m-i}}=a_1\dots a_m,
\]
i.e. \eqref{eq:cut-and-reflect} with $m-i$ in place of $i$, and so $m-i\geq i$ as well. Thus, $i=m-i$. But this means $m=2i$ and we have $a_{i+1}\dots a_m=\overline{a_1\dots a_i}$, so $\sa=\sb\overline{\sb}$ for some word $\sb$, contrary to the hypothesis of the lemma.
\end{proof}

\section{Proof of Theorem \ref{main2:Ga-topology}} \label{sec:Gamma}

Recall the definition of $\Ga_{AC}$ from \eqref{eq:Gamma-AC-def}.
The set $\Ga_{AC}$ was introduced by Allouche and Cosnard in 1983 \cite{Allouche_Cosnard_1983}, and its importance has since been demonstrated in a variety of settings. It is known to be of Lebesgue measure zero but of full Hausdorff dimension (see \cite{Bon-Car-Ste-Giu-2013}). The following lemma collects some known topological properties of $\Ga_{AC}$.

\begin{lemma} \label{lem:isolated-point-of-V}
The set $\Ga_{AC}$ is compact, and contains infinitely many isolated and infinitely many accumulation points. Furthermore,
\begin{enumerate}[{\rm(i)}]
  \item $2/3=\min \Ga_{AC}$ is an isolated point of $\Ga_{AC}$, and $1=\max \Ga_{AC}$ is an accumulation point of $\Ga_{AC}$;
	\item A point $t\in\Ga_{AC}$ is isolated in $\Ga_{AC}$ if and only if $D^n(t)=1-t$ for some $n\in\N$;
\end{enumerate}
\end{lemma}

\begin{proof}
These statements follow directly from \cite{Allouche-Cosnard-2001}, where it is shown that the isolated points of $\Ga_{AC}$ are obtained by ``period doubling"; i.e. $t\in\Ga_{AC}$ is an isolated point of $\Ga_{AC}$ if and only if $b(t)=(\sb\overline{\sb})^\f$ for a suitable word $\sb$ (see \cite{Allouche-Cosnard-2001}). Thus, $t$ is isolated in $\Ga_{AC}$ if and only if $D^n(t)=1-t$ for some $n$. In particular, taking $\sb=1$ we obtain the isolated point $2/3$, since $b(2/3)=(10)^\f$.
\end{proof}

\begin{corollary} \label{cor:isolated-point}
If $t$ is an isolated point of $\Ga_{AC}$, then $\tilde{\tau}(t)=1-t$, and hence $1-t\in\tGa$.
\end{corollary}

\begin{proof}
Let $t$ be an isolated point of $\Ga_{AC}$. Then by Lemma \ref{lem:isolated-point-of-V} (ii) there is an $n\in\N$ such that $D^n(t)=1-t$, so that $D^{2n}(t)=D^n(1-t)=t$. But then $D^{n(2k-1)}(t)=1-t$ for all $k\in\N$, so $\tilde{\tau}(t)\leq\liminf_{j\to\infty}D^j(t)\leq 1-t$. On the other hand, since $t\in \Ga_{AC}$, we have $\min\set{D^j(t), D^j(1-t)}\ge 1-t$ for all $j\ge 0$, and so $\tilde{\tau}(t)\geq 1-t$. Hence, $\tilde{\tau}(t)=1-t$.
\end{proof}

\begin{proposition}  \label{prop:char-jump}
  $\tGa=1-\Ga_{AC}$.
\end{proposition}

\begin{proof}
First take $t\in\tGa$. We need to show that
	\begin{equation} \label{eq:self-dominating}
	t\leq D^n(t)\leq 1-t \quad\forall n\geq 0.
	\end{equation}
	Fix $n\geq 0$. Since $t\in\tGa$, there exists $x\in \TT'=\TT\setminus\set{k 2^{-n}: k, n\in\N}$ such that $\tilde{\tau}(x)=t$. By symmetry of $\tilde{\tau}$ (i.e., $\tilde{\tau}(x)=\tilde{\tau}(1-x)$) we may assume $\tilde{\tau}(x)=\liminf_{k\to\f}D^k(x)=t$, as otherwise we can replace $x$ by $1-x$. So there is a subsequence $(n_i)$ such that $D^{n_i}(x)\to t$ as $i\to\f$.
  Using the continuity of $D^n$ on $\TT'$ it follows that
  \[
	t=\tilde{\tau}(x)\leq\liminf_{i\to\infty} D^{n+n_i}(x)=D^n(t),
  \]
	and similarly,
	\[
	t\leq \liminf_{i\to\infty} D^{n+n_i}(1-x)=1-\lim_{i\to\infty}D^{n+n_i}(x)=1-D^n(t).
	\]
	This proves \eqref{eq:self-dominating}, and it follows that $t\in 1-\Ga_{AC}$.

Next, take $t\in 1-\Ga_{AC}$. If $1-t$ is an isolated point of $\Ga_{AC}$, then $t\in\tGa$ by Corollary \ref{cor:isolated-point}. Suppose now that $1-t$ is an accumulation point of $\Ga_{AC}$. Then by Lemma \ref{lem:isolated-point-of-V} (ii), we have
\[
b(t)\prec \sigma^n(\overline{b(t)})\lle \overline{b(t)} \qquad\forall n\geq 0,
\]
so setting $\al:=\al_1\al_2\ldots:=\overline{b(t)}$, $\al$ satisfies the hypotheses of \cite[Lemma 4.1]{Komornik_Loreti_2007}, and hence there is a strictly increasing sequence $(m_j)$ such that for each $j$, the word $\al_1\dots\al_{m_j}^-$ is admissible (see Definition \ref{def:admissible}).
Now set
\[
(d_i):=\al_1\dots\al_{m_1}^- \al_1\dots\al_{m_2}^- \al_1\dots\al_{m_3}^-\dots,\qquad\textrm{and}\qquad x:=\pi_2((d_i)).
\]
Clearly,
\[
\limsup_{k\to\f}D^{m_1+m_2+\ldots+m_k}(x)\ge\lim_{k\to\f}\pi_2\big(\al_1\ldots \al_{m_{k+1}}^-0^\f\big)=\pi_2\big(\overline{b(t)}\big)=1-t,
\]
which implies $\tilde{\tau}(x)\le\liminf_{n\to\f}D^n (1-x)=1-\limsup_{n\to\f}D^n(x)\leq t$.  To prove the reverse inequality we claim that
\begin{equation}\label{eq:august-6-1}
b(t)\preceq \sigma^n((d_i))\preceq \overline{b(t)} \qquad\forall n\geq 0.
\end{equation}
Since $\al_1\ldots \al_{m_j}^-$ is admissible for each $j$, by Definition \ref{def:admissible} (applied first to $i$ and then to $m_j-i$ in place of $i$) it follows that
\begin{equation}\label{eq:august-6-2}
 \overline{\al_1\ldots \al_{m_j}}\prec \al_{i+1}\ldots \al_{m_j}^-\al_1\ldots \al_i\prec \al_1\ldots\al_{m_j}\quad\forall~ 0\le i<m_j.
\end{equation}
Since $m_{j+1}>m_j$, (\ref{eq:august-6-2}) implies (\ref{eq:august-6-1}), and this gives $\tilde{\tau}(x)\ge t$. As a result, $\tilde{\tau}(x)=t$ and therefore, $t\in\tGa$. 
\end{proof}

\begin{remark} \label{rem:char-Gamma}
Proposition \ref{prop:char-jump} and Lemma \ref{lem:isolated-point-of-V} imply that $\tGa$ is compact. As a consequence we can show that $\tGa$ is in fact the right bifurcation set of the set-valued map $t\mapsto \widetilde{S}(t)$:
\[
t\in\tGa\quad\Longleftrightarrow\quad \widetilde{S}(t')\ne \widetilde{S}(t)\quad\forall t'>t.
\]
This can be seen as follows. If $t\in\tGa$, then $\tilde\tau(x)=t$ for some $x\in [0,1)$, so $x\in \widetilde{S}(t)\setminus \widetilde{S}(t')$ for any $t'>t$. Conversely, take $t\notin \tGa$. Since $\tGa$ is compact, there exists $\ep>0$ such that $\tGa\cap[t, t+\ep)=\emptyset$. This implies that $\widetilde{S}(t)=\widetilde{S}(t')$ for any $t'\in(t, t+\ep)$.
\end{remark}

\begin{proof}[Proof of Theorem \ref{main2:Ga-topology}]
Immediate from Proposition \ref{prop:char-jump} and Lemma \ref{lem:isolated-point-of-V} (ii).
\end{proof}

\section{Entropy, survivor sets and bifurcation set for the doubling map} \label{sec:entropy}

For $0\leq\theta\leq 1/2$, let
\begin{equation} \label{eq:K-theta}
K(\theta):=\{x\in\mathbb{T}: D^n(x)\not\in(\theta,1-\theta)\ \forall n\geq 0\}.
\end{equation}
{The set $K(\theta)$ is known as the {\em survivor set} of the open dynamical system $(\mathbb{T},D)$ with the ``hole" $(\theta,1-\theta)$. That is, $K(\theta)$ is the set of points whose forward orbit under $D$ never enters the hole. Such open dynamical systems were first considered by Urba\'nski \cite{Urbanski-1986} in the setting of the unit circle, and there is an extensive literature about them.}

Note that $K(\theta)$ is forward invariant under the doubling map $D$, and the Lebesgue measure $m|_{\mathbb T}$ is the unique ergodic measure for the dynamical system $(\mathbb T, D)$ (cf.~\cite{Walters_1982}). Thus, it follows from the Birkhoff ergodic theorem that $K(\theta)$ has zero Lebesgue measure for any $\theta<1/2$. So it is natural to consider the Hausdorff dimension of $K(\theta)$ for $\theta\in[0,1/2)$. It is well-known (cf.~\cite{Raith_1989} or \cite{Tiozzo_2018}) that
\[
\dim_H K(\theta)=\frac{h(\theta)}{\log 2}\quad\forall\, 0\le\theta\le 1/2,
\]
where
\[
h(\theta):=h\big(D|_{K(\theta)}\big)
\]
is the topological entropy of $D$ on $K(\theta)$. See, for instance, \cite{Douady-1995} for the definition of topological entropy. The function $h(\theta)$ was studied in detail by Douady \cite{Douady-1995}, who used it to establish properties of the entropy of real quadratic polynomials. We collect here some known facts about the size of $K(\theta)$ and the entropy $h(\theta)$. First, recall from \cite{Allouche_Shallit_1999} the Thue-Morse sequence $(\tau_i)_{i=0}^\f\in\Omega$, defined recursively by
\begin{equation}\label{eq:thue-morse}
\tau_0=0,\quad\textrm{and}\quad\tau_{2^n}\ldots \tau_{2^{n+1}-1}=\overline{\tau_0\ldots\tau_{2^n-1}}\quad\forall ~n\ge 0.
\end{equation}
Then $(\tau_i)_{i=0}^\f=01101001\ldots$. The {\em Feigenbaum angle} is the number $\theta_F:=\pi_2(\tau_0\tau_1\ldots)\approx .412454\dots$.

\begin{theorem}[\cite{Douady-1995,Tiozzo_2018}] \label{thm:K-theta}
\mbox{}

\begin{enumerate}[{\rm (i)}]
\item $K(\theta)=\{0\}$ for $0\leq\theta<1/3$;
\item $K(\theta)$ is countably infinite for $1/3\leq\theta<\theta_F$;
\item $K(\theta)$ is uncountable but of zero Hausdorff dimension for $\theta=\theta_F$;
\item $K(\theta)$ is uncountable for $\theta_F<\theta\leq 1/2$, and furthermore,
\[
\dim_H K(\theta)=\frac{h(\theta)}{\log 2}>0 \qquad\mbox{for all $\theta>\theta_F$},
\]
and the function $\theta\mapsto h(\theta)$ is a devil's staircase: continuous, nondecreasing and locally constant almost everywhere.
\end{enumerate}
\end{theorem}

In Section \ref{sec:Gamma-dimension} we shall relate the sets $\widetilde{S}(t), t\in[0,1]$ to the sets $K(\theta), 0\leq\theta\leq 1/2$ via the reparametrization $t=1-2\theta$, so that the behavior of the Hausdorff dimension of $\widetilde{S}(t)$ may be deduced from Theorem \ref{thm:K-theta}.

Let $E^h$ denote the bifurcation set of the map $\theta\mapsto h(\theta)$:
\[
E^h:=\{\theta\in[0,1/2]: h(\theta-\ep)<h(\theta+\ep)\ \forall \ep>0\}.
\]
It follows from Theorem \ref{thm:K-theta} (iv) that $E^h$ has Lebesgue measure zero. We will show later, in Proposition \ref{prop:dimension-of-E}, that it has full Hausdorff dimension.

We also define a one-sided version of $E^h$, namely
\[
E^h_L:=\{\theta\in[0,1/2]: h(\theta-\ep)<h(\theta)\ \forall \ep>0\}.
\]
The maximal intervals of constancy (the ``entropy plateaus") of $h(\theta)$ were first characterized by Douady \cite[p.~86]{Douady-1995}, and were recently described more explicitly by Tiozzo \cite{Tiozzo_2015,Tiozzo_2018}. It will be convenient for us here to {use the admissible words from Definition \ref{def:admissible}} to describe these plateaus.

To each admissible word $\sa=a_1\dots a_m$ which is not of the form $\sb\overline{\sb}$ we associate a word $\ss=s_0\dots s_{m-1}:=0 a_1\dots a_{m-1}$ of the same length. To this word $\ss$, which we call a {\em tuning base}, we associate the closed interval $[\theta_L,\theta_R]:=[\theta_L(\ss),\theta_R(\ss)]\subseteq [0,1)$ such that $\theta_L=.\ss^\f$ and $\theta_R=.\ss \overline{\ss}^\f$. Some of these intervals are contained within others; however we have no need in this paper to distinguish which of them are maximal. What matters is that any two such intervals are either disjoint, or else one is contained in the other. The intervals $[\theta_L,\theta_R]$ are called {\em tuning windows}; see \cite[end of section 1]{Tiozzo_2018}. Now
\begin{equation} \label{eq:Eh}
E^h=[\theta_F,1/2]\backslash \bigcup_{\ss} (\theta_L(\ss),\theta_R(\ss)),
\end{equation}
where the union is over all tuning bases $\ss$. It follows in particular that $E^h$ is a Cantor set. Similarly,
\begin{equation} \label{eq:EhL}
E^h_L=(\theta_F,1/2]\backslash \bigcup_{\ss} (\theta_L(\ss),\theta_R(\ss)].
\end{equation}

Next, considering that the set valued map $\theta\mapsto K(\theta)$ is non-decreasing, we introduce the bifurcation set $\cR$ of the set-valued map $\theta\mapsto K(\theta)$. For definiteness, we set $K(\theta)=\emptyset$ for $\theta<0$, and $K(\theta)=[0,1)$ for $\theta\geq 1/2$. Let
\begin{equation} \label{eq:definition-of-R}
\cR:=\{\theta\in[0,1/2]: K(\theta-\ep)\neq K(\theta+\ep)\ \forall \ep>0\}.
\end{equation}
Clearly, $E^h\subseteq \cR$.

\begin{lemma} \label{lem:R}
We have
\[
\cR=\{\theta\in[0,1/2]: D^n(\theta)\not\in(\theta,1-\theta)\ \forall n\geq 0\}=\{\theta\in[0,1/2]: \theta\in K(\theta)\}.
\]
\end{lemma}

\begin{proof}
Note first that $K(0)=\{0\}$ and $K(-\ep)=\emptyset$ for all $\ep>0$, so $0\in\cR$ and $0\in K(0)$. In the following, we fix $\theta\in(0,1/2]$.

If $\theta\in K(\theta)$, then for each $\theta'<\theta$ we have $\theta\in K(\theta)\backslash K(\theta')$, so $\theta\in\cR$.

Conversely, suppose $\theta\not\in K(\theta)$. Then there is $n_0\in\N$ such that $D^{n_0}(\theta)\in(\theta,1-\theta)$. Since $D^{n_0}$ is continuous and $D^{n_0}(\theta)\neq 0$, there is an $\ep>0$ such that $D^{n_0}$ is strictly increasing on $(\theta-\ep,\theta+\ep)$, and $\theta+\ep<D^{n_0}(\theta-\ep)<D^{n_0}(\theta+\ep)<1-\theta-\ep$. Hence,
\begin{equation} \label{eq:in-forbidden-interval}
D^{n_0}(x)\in (\theta+\ep,1-\theta-\ep) \qquad\forall x\in [\theta-\ep,\theta+\ep].
\end{equation}
Now suppose $x\notin K(\theta-\ep)$. Then $D^j(x)\in(\theta-\ep, 1-\theta+\ep)$ for some $j\ge 0$. So either (i) $D^j(x)\in(\theta+\ep,1-\theta-\ep)$; or (ii) $D^j(x)\in(\theta-\ep,\theta+\ep]$, in which case $D^{n_0+j}(x)\in(\theta+\ep,1-\theta-\ep)$ by \eqref{eq:in-forbidden-interval}; or (iii) $D^j(x)\in[1-\theta-\ep,1-\theta+\ep)$, in which case $D^j(1-x)\in(\theta-\ep,\theta+\ep]$ so that $D^{n_0+j}(1-x)\in (\theta+\ep,1-\theta-\ep)$ by \eqref{eq:in-forbidden-interval}, and then also $D^{n_0+j}(x)\in (\theta+\ep,1-\theta-\ep)$. In all three cases, we conclude that $x\notin K(\theta+\ep)$, and thus $K(\theta+\ep)\subset K(\theta-\ep)$. Since the map $\theta\mapsto K(\theta)$ is non-decreasing, we obtain $K(\theta-\ep)=K(\theta+\ep)$, and therefore $\theta\not\in\cR$.
\end{proof}

It is not difficult to see that $\theta\in K(\theta)$ if and only if $\theta=0$ or $2\theta\in \Ga_{AC}$; hence we have
\begin{equation} \label{eq:R-identity}
\cR\backslash\{0\}=\frac12\Ga_{AC}.
\end{equation}
(This identity is well known; see for instance \cite{Tiozzo_2018} or \cite{Bon-Car-Ste-Giu-2013}.)
The set $\cR$ has several other interpretations: It is also the set of external angles of rays landing on the real slice of the Mandelbrot set \cite{Tiozzo_2015}, or the set of angles parametrizing real quadratic minor laminations \cite{Bon-Car-Ste-Giu-2013}.

\subsection{The tuning map and its properties} \label{subsec:Psi}

Recall that to each admissible word $\sa=a_1\dots a_m$ not of the form $\mathbf{b}\overline{\mathbf{b}}$ corresponds a tuning window $I=I(\ss)=[\theta_L,\theta_R]$ given by $\theta_L=.\ss^\f$ and $\theta_R=.\ss\overline{\ss}^\f$, where $\ss=0a_1\dots a_{m-1}$, and the entropy $h(\theta)$ is constant on $[\theta_L,\theta_R]$. To this tuning window we associate a map $\Phi_I: \Omega\to\Omega$ defined by
\[
\Phi_I((x_i))=\Sigma_{x_1}\Sigma_{x_2}\Sigma_{x_3}\dots, \qquad (x_i)\in\Omega,
\]
where
\[
\Sigma_0:=\ss, \qquad \Sigma_1:=\overline{\ss}.
\]
Now set
\begin{equation} \label{eq:Psi-def}
\Psi_I:=\pi_2\circ \Phi_I\circ \pi_2^{-1}.
\end{equation}
(For definiteness, when $\theta\in[0,1)$ has two binary expansions, we let $\pi_2^{-1}(\theta)$ be the one ending in $0^\f$.) We call $\Psi_I$ a {\em tuning map}. {To our knowledge, the above explicit form of the tuning map was first given by Douady \cite[Section 4.6]{Douady-1995}.} Since $\ss\prec\overline{\ss}$, we see that $\Psi_I$ is strictly increasing, and it maps $[0,1/2]$ into $I$. Furthermore, $\Psi_I(\cR)=\cR\cap I$ (see the proof of \cite[Proposition 9.3]{Tiozzo_2015}).

\begin{lemma} \label{lem:Psi-Holder}
Let $I=I(\ss)$ be a tuning window. Then for each $k\in\N$, the restriction of $\Psi_I$ to $\pi_2(\mathbf{X}_k)$ is bi-H\"older continuous with exponent $|\ss|$.
\end{lemma}

\begin{proof}
This follows from Lemma \ref{lem:bi-Lip} and \eqref{eq:Psi-def}, since the map $\Phi_I$ is bi-H\"older continuous with exponent $|\ss|$ under the metric $\gamma$ defined in (\ref{eq:metric}).
\end{proof}

\begin{corollary} \label{cor:Psi-Holder-on-R}
Let $I=I(\ss)$ be a tuning window. For each $\ep>0$, the restriction of $\Psi_I$ to $\cR\cap[0,1/2-\ep]$ is bi-H\"older continuous with exponent $|\ss|$.
\end{corollary}

\begin{proof}
Given $\ep>0$, choose $k\in\N$ so that $2^{-k}<\ep$. Then Lemma \ref{lem:R} implies $\cR\cap[0,1/2-\ep]\subseteq \pi_2(\mathbf{X}_k)$. {(Note that $0^\f\in \mathbf{X}_k$.)} Thus we are done by applying Lemma \ref{lem:Psi-Holder}.
\end{proof}

For a tuning window $I$ and $\theta\in[0,1/2]$, set
\[
K_I(\theta):=K(\theta)\cap I.
\]
The following basic fact about the tuning map will be of crucial importance in Section \ref{sec:level-set}.

\begin{lemma} \label{lem:K-under-tuning}
Let $I=I(\ss)$ be a tuning window generated by the word $\ss$.
Let $\hat{\theta}\in\cR$, and $\theta=\Psi_I(\hat{\theta})$. Then
\begin{equation} \label{eq:K-renormalization}
K_I(\theta)=\Psi_I\big(K(\hat{\theta})\cap[0,1/2]\big),
\end{equation}
and so
\begin{equation} \label{eq:K-theta-tuning-dimension}
\dim_H K_I(\theta)=\frac{1}{|\ss|}\dim_H K(\hat{\theta}).
\end{equation}
\end{lemma}

\begin{proof}
This can be essentially deduced from \cite[Section 9.1]{Tiozzo_2015}, but for the reader's convenience we include a short proof. Let $x=.x_1x_2x_3\dots\in K_I(\theta)$. Then $.\ss^\f\leq x\leq .\ss\overline{\ss}^\f$, so $x_1\dots x_m=\ss$, where $m:=|\ss|$. Observe that $(x_i)$ cannot contain a block $\ss\sb$ with $\sb\succ\overline{\ss}$, for otherwise there would be an integer $n$ such that $D^n(x)\geq .\ss\sb 0^\f>.\ss\overline{\ss}^\f\ge\theta$ and $D^n(x)\leq .\ss 1^\f\leq 1/2<1-\theta$, contradicting that $x\in K(\theta)$. Similarly, $(x_i)$ cannot contain a block $\overline{\ss}\sb$ with $\sb\prec\ss$. Since furthermore, $x_{n+1}\dots x_{n+m}\lle\ss$ or $x_{n+1}\dots x_{n+m}\lge\overline{\ss}$ for all $n\geq 0$, it follows that each block $\ss$ or $\overline{\ss}$ in $(x_i)$ can only be followed by another block $\ss$ or $\overline{\ss}$. Thus, $x=\Psi_I(\hat{x})$ for some $\hat{x}$. Since $\Psi_I$ is strictly increasing with $\Psi_I(\hat{\theta})=\theta$ and $\Psi_I(1-\hat{\theta})=1-\theta$, it follows that in fact $\hat{x}\in K(\hat{\theta})$. Finally, since $(x_i)$ begins with $\ss$, we have $\hat{x}\in[0,1/2]$. Thus,
\[
K_I(\theta)\subseteq\Psi_I\big(K(\hat{\theta})\cap[0,1/2]\big).
\]
The reverse inclusion follows similarly, hence we have \eqref{eq:K-renormalization}. Now for $\hat{\theta}<1/2$, there is a positive integer $k$ such that for each $x\in K(\hat{\theta})\cap[0,1/2]$, the binary expansion of $x$ does not contain $k$ consecutive 1's and does not contain $k$ consecutive 0's after the first 1. Hence $K(\hat{\theta})\cap[0,1/2]\subseteq \pi_2(\mathbf{X}_k)$, and Lemma \ref{lem:Psi-Holder} implies that the restriction of $\Psi_I$ to $K(\hat{\theta})\cap[0,1/2]$ is bi-H\"older continuous with exponent $|\ss|$. Thus, \eqref{eq:K-theta-tuning-dimension} follows from \eqref{eq:K-renormalization} and the symmetry of $K(\hat{\theta})$ about $1/2$.
\end{proof}

\begin{lemma} \label{lem:K_I-continuity}
For a tuning window $I=I(\ss)=[\theta_L,\theta_R]$, the function $\theta\mapsto \dim_H K_I(\theta)$ is continuous in $(\theta_L,\theta_R)$.
\end{lemma}

\begin{proof}
Since $K_I(\theta)$ is constant on each connected component of $I\backslash\cR$, it suffices to consider $\theta\in\cR\cap (\theta_L,\theta_R)$. We prove right continuity; the argument for left continuity is the same.

If $\theta$ is isolated in $\cR$ from the right, then $K_I(\theta')=K_I(\theta)$ for all $\theta'>\theta$ sufficiently close to $\theta$, and right continuity follows.

Otherwise, we can find a sequence $(\theta_n)$ in $\cR\cap(\theta_L,\theta_R)$ such that $\theta_n\searrow \theta$, and it is sufficient to show that $\dim_H K_I(\theta_n)\to \dim_H K_I(\theta)$, in view of the monotonicity of the set-valued map $\theta'\mapsto K_I(\theta')$.
Let $\hat{\theta}:=\Psi_I^{-1}(\theta)$ and $\hat{\theta}_n:=\Psi_I^{-1}(\theta_n)$. Then $\hat{\theta}_n\searrow \hat{\theta}$, so by \eqref{eq:K-theta-tuning-dimension} and the continuity of $\theta'\mapsto \dim_H K(\theta')$,
\[
\dim_H K_I(\theta_n)=\frac{1}{|\ss|}\dim_H K(\hat{\theta}_n)\to \frac{1}{|\ss|}\dim_H K(\hat{\theta})=\dim_H K(\theta),
\]
completing the proof.
\end{proof}

Next, let $\CC^h$ denote the set of infinitely renormalizable angles, i.e. those $\theta\in [0,1/2)$ that belong to infinitely many tuning windows. It is clear that $\CC^h$ is uncountable, and not hard to show that it has zero Hausdorff dimension (see, for instance, \cite[Proposition 1.4]{Allaart-Kong-2018}). Furthermore, $\CC^h\subset\mathcal{R}$.

Recall from (\ref{eq:Eh}) that $E^h$ is the bifurcation set of the map $\theta\mapsto\dim_H K(\theta)$. Now for a tuning window $I=I(\ss)$, we define the {\em relative bifurcation set}
\[
E^h(I):=\{\theta\in I: \dim_H K_I(\theta-\ep)<\dim_H K_I(\theta+\ep)\ \ \forall \ep>0\}.
\]
It follows from Lemma \ref{lem:K-under-tuning} that
\begin{equation} \label{eq:bifurcation-tuning}
E^h(I)=\Psi_I(E^h).
\end{equation}

\begin{lemma} \label{lem:R-decomposition}
\begin{equation} \label{eq:R-decomposition}
\cR=\cR_{iso}\cup \CC^h \cup E^h\cup \bigcup_I E^h(I),
\end{equation}
with the union disjoint, where $\cR_{iso}$ denotes the set of isolated points of $\cR$, and the last union is over all tuning windows $I$.
\end{lemma}

\begin{proof}
For a tuning window $I=I(\ss)$, the tuning windows properly contained in $I$ are precisely the intervals $[\theta_L(\mathbf{t}),\theta_R(\mathbf{t})]=[.\mathbf{t}^\f,.\mathbf{t}\overline{\mathbf{t}}^\f]$, where $\mathbf{t}^\f=\Phi_I(\mathbf{u}^\f)$ and $\mathbf{t}\overline{\mathbf{t}}^\f=\Phi_I(\mathbf{u}\overline{\mathbf{u}}^\f)$ for some tuning base $\mathbf{u}$. (See \cite{Douady-1995}.) Thus, by \eqref{eq:bifurcation-tuning} and the strict monotonicity of $\Psi_I$, $I\backslash E^h(I)$ consists of the interval $[\theta_L(\ss),\Psi_I(\theta_F))$ and the interiors of the tuning intervals properly contained in $I$. Furthermore, it is well known that all points of $\cR$ in $[0,\theta_F)$ belong to $\cR_{iso}$, so the same is true for all points of $\cR$ in $[\Psi_I(0),\Psi_I(\theta_F))=[\theta_L(\ss),\Psi_I(\theta_F))$, since $\Psi_I: \cR\mapsto\cR\cap I$ is a homeomorphism. These observations yield \eqref{eq:R-decomposition}. That the union is disjoint also follows easily; for instance, if $I$ and $J$ are tuning windows with $J\subsetneq I$, then $E^h(J)\subseteq J$ whereas $E^h(I)$ is disjoint from $J$, so $E^h(I)\cap E^h(J)=\emptyset$. Finally, $E^h\cap \cR_{iso}=\emptyset$ since $E^h$ does not have isolated points by the continuity of the entropy function $h(\theta)$; and similarly, $E^h(I)\cap \cR_{iso}=\emptyset$ for each tuning window $I$.
\end{proof}

\section{The local dimension of $\cR$} \label{sec:R-local-dimension}

The following analogue of Theorem \ref{main:Ga-dimension} is essentially Theorem 1.7 of Tiozzo \cite{Tiozzo_2015}. Since Tiozzo's proof is spread over several sections and involves many concepts, we include a more condensed proof here for the reader's convenience, also because the construction of the ``reset block" in the proof below will be needed again in Section \ref{sec:level-set}. Note also that, while Tiozzo considers iteration of points under the tent map, we prefer here to work in the symbol space, since the constructions in Section \ref{sec:level-set} are given by concatenating infinitely many finite words.

\begin{theorem} \label{thm:R-and-K}
For each $\theta\in[0,1/2]$ we have
\[
\dim_H (\cR\cap[0,\theta])=\dim_H K(\theta).
\]
\end{theorem}

\begin{proof}
One inequality is clear: From Lemma \ref{lem:R} and the definition of $K(\theta)$ we see that $\cR\cap[0,\theta]\subseteq K(\theta)$, and hence $\dim_H(\cR\cap[0,\theta])\leq \dim_H K(\theta)$.

The reverse inequality is more involved. Note by Theorem \ref{thm:K-theta} that $\dim_H K(\theta)=0$ for any $\theta\in[0, \theta_F]$. Furthermore, by (\ref{eq:EhL}) the map $\theta\mapsto\dim_H K(\theta)$ is constant in each connected component of $(\theta_F, 1/2]\setminus E^h_L$. Hence without loss of generality we may assume that $\theta\in E^h_L$. Then $\theta\in\cR$, so $2\theta\in \Ga_{AC}$ by \eqref{eq:R-identity}, but $2\theta$ is not an isolated point of $\Ga_{AC}$, in view of \eqref{eq:EhL} and Lemma \ref{lem:isolated-point-of-V}. Let $\alpha=(\alpha_i):=b(2\theta)$. Thus $\overline{\alpha}\prec \sigma^n(\alpha)\lle \alpha$ for all $n\geq 0$, so $\al$ satisfies the hypotheses of \cite[Lemma 4.1]{Komornik_Loreti_2007}, and hence there is a strictly increasing sequence $(m_j)$ such that for each $j$, the word $\al_1\dots\al_{m_j}^-$ is admissible; that is,
\begin{equation} \label{eq:admissible}
\overline{\al_1\ldots \al_{m_j-i}}\lle \al_{i+1}\ldots \al_{m_j}^-\prec \al_1\ldots \al_{m_j-i}\quad\forall ~1\le i<m_j.
\end{equation}

Let $0<\theta'<\theta$. We will show that $\cR\cap[0,\theta]$ contains a Lipschitz copy of the set
\[
\hat{K}(\theta'):=K(\theta')\cap[1/2,3/4).
\]
Since $b(2\theta')\prec b(2\theta)=(\alpha_i)$, we can find $j$ large enough so that $m:=m_j$ satisfies \eqref{eq:admissible} and
\begin{equation} \label{eq:august-6-3}
 \al_1\dots\al_{m-1} 0^\f\succ b(2\theta').
\end{equation}

Next, since $\theta\in E^h_L$ implies $\theta>\theta_F$, we have $\al=b(2\theta)\succ b(2\theta_F)=\tau_1\tau_2\dots=1101\dots$. So there is an index $l_0\ge 3$ such that $\al_i=\tau_i$ for $1\leq i\leq l_0-1$, and $\al_{l_0}>\tau_{l_0}$.

Now set $i_0=m$, and for $\nu=0,1,2,\dots$, proceed inductively as follows. If $i_\nu<l_0$, then stop. Otherwise, let $i_{\nu+1}$ be the largest integer $i$ such that
\[
\overline{\al_{i_\nu-i+1}\dots \al_{i_\nu}}=\al_1\dots\al_i^-,
\]
or $i_{\nu+1}=0$ if no such $i$ exists. It is easy to check that $i_{\nu+1}<i_\nu$ for each $\nu$, so this process will stop after some finite number $N$ of steps, with $i_N<l_0$. One can check that $\al_1\dots\al_{i_\nu}^-$ is admissible for each $\nu=1,2,\dots,N-1$, i.e., each word $\al_1\ldots \al_{i_\nu}^-$ satisfies the inequalities in (\ref{eq:admissible}).

For $\nu=0,1,\dots,N-1$, we now argue as follows. Let $\ss_\nu:=0\alpha_1\dots\alpha_{i_\nu-1}$. Since $\alpha_{i_\nu}=1$,
\[
b(\theta)=0\alpha_1\alpha_2\ldots\succ (0\alpha_1\dots\alpha_{i_\nu-1})^\f=\ss_\nu^\f,
\]
and since $\theta\in E^h_L$, it follows from \eqref{eq:EhL} that $b(\theta)\succ \ss_\nu\overline{\ss_\nu}^\f$, so there is a positive integer $k_\nu$ such that
\begin{equation} \label{eq:k-nu-inequality}
\alpha_1\dots\alpha_{i_\nu(k_\nu+1)-1}\succ \alpha_1\dots\alpha_{i_\nu-1}(1\overline{\alpha_1\dots\alpha_{i_\nu-1}})^{k_\nu}.
\end{equation}
Put $W_\nu:=\ss_\nu^{k_\nu}$. Finally, set $R:=W_0 W_1\dots W_{N-1}0$. 
Let $x_0:=\pi_2(R0^\f)$ be the dyadic rational determined by the block $R$, and let $|R|$ denote the length of $R$. We claim that
\begin{equation} \label{key-implication}
x_0+2^{-|R|}\hat{K}(\theta') \subseteq \cR\cap [0,\theta].
\end{equation}

Let $x\in \hat{K}(\theta')$ and set $y:=x_0+2^{-|R|}x$. Note by the construction of $R$ that $y<\theta$.
It remains to verify that $y \in \cR$, or equivalently,
\begin{equation} \label{eq:y-self-sandwich}
D^n(y)\not\in(y,1-y) \qquad\forall n\geq 0.
\end{equation}
For $n\geq |R|$, these inequalities follow immediately since $D^{|R|}(y)=x \in K(\theta')\subseteq K(y)$, where the last inclusion uses that $y>\pi_2(W_0 0^\f)>\theta'$ by \eqref{eq:august-6-3}.

For $n<|R|$, \eqref{eq:y-self-sandwich} follows from (\ref{eq:august-6-3}), (\ref{eq:k-nu-inequality}) and the admissibility of $\al_1\dots\al_{i_\nu}^-$ for $\nu=0,1,\dots,N-1$, along with the fact that for $x\in \hat{K}(\theta')$, $b(x)$ begins with $10$; see \cite[Proposition 3.17]{AlcarazBarrera-Baker-Kong-2016} for the details.

Having established \eqref{key-implication}, we finish the proof by observing that the sets $\{x\in K(\theta'): b(x)\ \mbox{begins with}\ 1^n 0\}$ and $\{x\in K(\theta'): b(x)\ \mbox{begins with}\ 0^n 1\}$, for $n=1,2,\dots$, partition $K(\theta')$ and all have the same Hausdorff dimension by the symmetry of $K(\theta')$. Since $\hat{K}(\theta')$ is one of these sets, it follows that $\dim_H \hat{K}(\theta')=\dim_H K(\theta')$. Thus, $\dim_H(\cR\cap[0,\theta])\geq \dim_H K(\theta')$. Letting $\theta'\nearrow\theta$ and using Theorem \ref{thm:K-theta}, it follows that
$\dim_H(\cR\cap[0,\theta])\geq \dim_H K(\theta)$.
\end{proof}

Using Theorem \ref{thm:R-and-K} we can calculate the local dimension of $\cR$:

\begin{proposition} \label{cor:R-level-1}
If $\theta\in E^h$, then
\[
\lim_{\ep\to 0}\dim_H\big(\cR\cap(\theta-\ep,\theta+\ep)\big)=\dim_H K(\theta).
\]
\end{proposition}

\begin{proof}
Take $\theta\in E^h$. Then for all $\ep>0$ we have $\dim_H K(\theta+\ep)>\dim_H K(\theta-\ep)$. Furthermore, by Theorem \ref{thm:R-and-K},
\begin{gather*}
\dim_H\big(\cR\cap[0,\theta-\ep]\big)=\dim_H K(\theta-\ep),\\
\dim_H\big(\cR\cap [0,\theta+\ep]\big)=\dim_H K(\theta+\ep).
\end{gather*}
Hence, using that $\dim_H(A\cup B)=\max\{\dim_H A,\dim_H B\}$, we conclude that
\[
\dim_H\big(\cR\cap(\theta-\ep,\theta+\ep)\big)=\dim_H K(\theta+\ep)\to \dim_H K(\theta) \quad\mbox{as $\ep\searrow 0$},
\]
where the final convergence follows from Theorem \ref{thm:K-theta} (iv). This completes the proof.
\end{proof}

We can extend the last result further by using the tuning maps $\Psi_I$.

\begin{proposition} \label{prop:R-local-dimension}
Let $\theta\in\mathcal{R}\backslash E^h$. Then
\begin{equation} \label{eq:gamma-local-dimension}
\lim_{\ep\to 0}\dim_H\big(\mathcal{R}\cap(\theta-\ep, \theta+\ep)\big)=\begin{cases}
0 & \mbox{if $\theta\in\cR_{iso}\cup\CC^h$},\\
\frac{1}{|\ss|}\dim_H K(\hat{\theta}) & \mbox{otherwise},
\end{cases}
\end{equation}
where in the second case, $I=I(\ss)=[\theta_L,\theta_R]$ is the smallest tuning window such that $\theta\in(\theta_L,\theta_R)$, and $\hat{\theta}:=\Psi_I^{-1}(\theta)$.
\end{proposition}

\begin{proof}
If $\theta\in\cR_{iso}$, then clearly $\dim_H(\cR\cap(\theta-\ep, \theta+\ep))=0$ for sufficiently small $\ep>0$.

Suppose $\theta\notin E^h\cup \CC^h\cup\cR_{iso}$. Then by \eqref{eq:R-decomposition} there is a unique tuning window
$I=I(\ss)=[\theta_L, \theta_R]$ such that $\theta\in E^h(I)$, and thus $\hat{\theta}=\Psi_I^{-1}(\theta)\in E^h$. 
So by Proposition \ref{cor:R-level-1} and Corollary \ref{cor:Psi-Holder-on-R} it follows that
\begin{equation} \label{eq:july18-1}
\begin{split}
  \lim_{\ep\to 0}\dim_H\big(\cR\cap(\theta-\ep, \theta+\ep)\big)&=\lim_{\ep\to 0}\dim_H\big(\cR\cap I\cap(\theta-\ep, \theta+\ep)\big)\\
  &=\lim_{\eta\to 0}\dim_H\Psi_I\big(\cR\cap(\hat \theta-\eta, \hat \theta+\eta)\big)\\
  &=\frac{1}{|\ss|}\lim_{\eta\to 0}\dim_H\big(\cR\cap(\hat \theta-\eta, \hat \theta+\eta)\big)\\
	&=\frac{1}{|\ss|}\dim_H K(\hat \theta).
\end{split}
\end{equation}
This proves \eqref{eq:gamma-local-dimension} for $\theta\in\cR\setminus(E^h\cup\CC^h\cup\cR_{iso})$.

Finally, let $\theta\in\CC^h$. Then $\theta$ lies in infinitely many tuning windows $I=I(\ss)$, so (\ref{eq:july18-1}) holds for infinitely many tuning bases $\ss$ and corresponding $\hat \theta\in E^h$. Since $\dim_H K(\hat \theta)\le 1$ for any $\hat \theta\in E^h$, it follows that $\lim_{\ep\to 0}\dim_H(\cR\cap(\theta-\ep, \theta+\ep))=0$, completing the proof.
\end{proof}

\section{Proof of Theorem \ref{main2:Ga-dimension}} \label{sec:Gamma-dimension}

In this section we consider the set $\wS(t)=\set{x\in[0,1): \tilde\tau(x)\ge t}$ for $t\in[0,1]$, and prove Theorem \ref{main2:Ga-dimension}. It will be convenient to introduce the sets
\begin{equation*} 
\wB(t):=\{x\in[0,1): D^n(x)\in[t,1-t]\ \forall n\geq 0\}, \qquad t\in[0,1],
\end{equation*}
where we interpret $\wB(t)$ to be empty when $t>1/2$. We will show in Proposition \ref{prop:S_tau-equal-dimension} that $\dim_H\wB(t)=\dim_H\wS(t)$ for all $t\in[0,1]$.

The sets $\wB(t)$ are closely related to the sets $K(\theta)$ from \eqref{eq:K-theta}. Namely, for $1/3\leq\theta\leq 1/2$, we have
\begin{equation} \label{eq:B-K-inclusions}
\wB(1-2\theta)\subseteq K(\theta)\subseteq \{0\}\cup \bigcup_{n=0}^\infty D^{-n}\big(\wB(1-2\theta)\big).
\end{equation}
(Both inclusions are proper {for $\theta<1/2$}: the first since $0\in K(\theta)\backslash \wB(1-2\theta)$, and the second since $\bigcup_{n=0}^\infty D^{-n}\big(\wB(1-2\theta)\big)$ is dense in $[0,1)$ while $K(\theta)$ is nowhere dense.) 
The inclusions \eqref{eq:B-K-inclusions}, which are easy to verify, were essentially proved in \cite[Proposition 3.1]{Alcaraz_Barrera_2014}. They imply that
\begin{equation} \label{eq:A-B-relation}
\dim_H K(\theta)=\dim_H \wB(1-2\theta), \qquad \theta\in [1/3,1/2].
\end{equation}
But the equality holds also for $0\leq\theta<1/3$, since $K(\theta)=\{0\}$ and $\wB(1-2\theta)=\emptyset$ in this case.

Recall from (\ref{eq:thue-morse}) the Thue-Morse sequence $(\tau_i)_{i=0}^\f=01101001\ldots$. Set
\begin{equation}\label{eq:tkl}
\tilde{t}_F:=\pi_2(\overline{\tau_1\tau_2\ldots})=1-\pi_2(\tau_1\tau_2\ldots)=1-2\theta_F\approx .1751.
\end{equation}
From \eqref{eq:B-K-inclusions}, \eqref{eq:A-B-relation} and Theorem \ref{thm:K-theta} we immediately obtain:

\begin{lemma} \label{lem:survivor-dimension}
  The map $\varphi: t\mapsto \dim_H \wB(t)$ is a non-increasing devil's staircase, i.e., it is non-increasing, continuous, locally constant almost everywhere in $[0,1]$, and $\varphi(0)>\varphi(1)$. Furthermore,
	\begin{enumerate}[{\rm(i)}]
    \item If $1/3<t\le 1$, then $\wB(t)=\emptyset$.

    \item If $\tilde{t}_F<t\le 1/3$, then $\set{1/3,2/3}\subseteq \wB(t)$ and $\wB(t)$ is at most countable.

    \item If $t=\tilde{t}_F$, then $\wB(t)$ is uncountable and $\dim_H \wB(t)=0$.

    \item If $0\le t<\tilde{t}_F$, then $\dim_H \wB(t)>0$.
  \end{enumerate}
\end{lemma}

We can now deduce similar properties of the sets $\wS(t)$, using the following inclusions.

\begin{lemma} \label{S_tau-inclusions}
We have for each $t\in(0,1]$ the inclusions
\begin{equation} \label{eq:inclusion}
\wB(t)\subseteq \wS(t)\subseteq\bigcup_{n=0}^\f D^{-n}(\wB(t'))\quad\forall\, 0\leq t'<t.
\end{equation}
Moreover, the first inclusion holds also for $t=0$.
\end{lemma}

\begin{proof}
First fix $t\in[0,1]$, and take $x\in \wB(t)$. Then $D^n(x)\geq t$ and $D^n(1-x)\geq t$ for all $n\geq 0$, so
\[
\tilde{\tau}(x)=\min\set{\liminf_{n\to\f}D^n(x), ~\liminf_{n\to\f}D^n(1-x)}\ge t.
\]
In other words, $x\in \wS(t)$. This proves the first inclusion.

Next, fix $t\in(0,1]$ and $0\leq t'<t$. Take $x\in \wS(t)$. Then $\tilde{\tau}(x)\geq t>t'$, so there exists $N\in\N$ such that $D^n(x)>t'$ and $D^n(1-x)>t'$ for all $n\ge N$. This implies that $D^N(x)\in \wB(t')$, and thus $x\in D^{-N}(\wB(t'))$. This proves the second inclusion.
\end{proof}

\begin{proposition} \label{prop:S_tau-equal-dimension}
For any $t\in[0,1]$ we have
 \[
  \dim_H \wS(t)=\dim_H \wB(t).
 \]
\end{proposition}

\begin{proof}
This follows from Lemma \ref{S_tau-inclusions}, the countable stability of Hausdorff dimension, and the continuity of $t\mapsto\dim_H \wB(t)$ from Lemma \ref{lem:survivor-dimension}, since for each $n\geq 0$, $D^{-n}(\wB(t'))$ is a finite union of affine images of $\wB(t')$.
\end{proof}

\begin{theorem} \label{th:survivor set}
The map $\varphi: t\mapsto \dim_H \wS(t)$ is a non-increasing devil's staircase on $[0, 1]$, i.e., $\varphi$ is  continuous and locally constant almost everywhere in $[0, 1]$, and $\varphi(0)>\varphi(1)$. Furthermore:
\begin{enumerate}[{\rm(i)}]
\item If $1/3<t\le 1$, then $\wS(t)=\emptyset$.

\item If $\tilde{t}_F<t\le 1/3$, then $\wS(t)$ is countably infinite.

\item If $t=\tilde{t}_F$, then $\wS(t)$ is uncountable and $\dim_H \wS(t)=0$.

\item If $0\le t<\tilde{t}_F$, then $\varphi(t)=\dim_H \wS(t)>0$. Furthermore, $\varphi(t)<\varphi(0)=1$ for any $t>0$.

\end{enumerate}
\end{theorem}

\begin{proof}
By Proposition \ref{prop:S_tau-equal-dimension} and Lemma \ref{lem:survivor-dimension} it follows that the map $t\mapsto\dim_H \wS(t)$ is a non-increasing devil's staircase.
So it only remains to prove the four items (i)--(iv). If $x\in \wB(t)$, then $x/2^k\in \wS(t)$ for any $k\in\N$.
Hence, the items (i)--(iv) follow from Lemma \ref{S_tau-inclusions} and the corresponding four items of Lemma \ref{lem:survivor-dimension}. Note that to get $\dim_H \wS(\tilde{t}_F)=0$ we used the continuity of $t\mapsto \dim_H \wB(t)$ from Lemma \ref{lem:survivor-dimension}.
\end{proof}

From Theorem \ref{thm:R-and-K} we next deduce the analogue of Theorem \ref{main:Ga-dimension} for the doubling map:

\begin{corollary} \label{prop:gamma-and-B}
For each $t\in (0,1]$,
\[
\dim_H(\tGa\cap[t,1])=\dim_H \wS(t).
\]
\end{corollary}

\begin{proof}
Given $t\in(0,1]$, let $\theta=(1-t)/2$, so $t=1-2\theta$. Since $\tGa=1-\Ga_{AC}=1-2(\cR\setminus\set{0})$, we have
\[
\dim_H(\tGa\cap[t,1])=\dim_H(\cR\cap[0,\theta])=\dim_H K(\theta)=\dim_H \wB(t)=\dim_H \wS(t),
\]
where the first equality follows since Hausdorff dimension is invariant under scale, and the other equalities follow successively from Theorem \ref{thm:R-and-K}, \eqref{eq:A-B-relation}, and Proposition \ref{prop:S_tau-equal-dimension}. 
\end{proof}

\begin{proof}[Proof of Theorem \ref{main2:Ga-dimension}]
Statement (ii) follows from Corollary \ref{prop:gamma-and-B}; (iii) follows from (ii) and Theorem \ref{th:survivor set}; and (i) follows from (ii) and (iii) since
\[
\dim_H\tGa\geq \lim_{t\searrow 0}\dim_H\big(\tGa\cap[t,1]\big)=\lim_{t\searrow 0}\dim_H\wS(t)=\dim_H\wS(0)=1.
\]
\end{proof}

\section{The level sets of $\tilde\tau$ and local dimension of $\tGa$} \label{sec:level-set}

In this section we consider the level sets $\wL(t)$ for $t\in\tGa$, and prove Theorem \ref{main2:level-set}. {In the process we also compute $\dim_H \wL(t)$ for $t\in\tGa\backslash \wE$.}
We let $\tGa_{iso}$ denote the set of all isolated points of $\tGa$.

\begin{proposition} \label{prop:level-set-isolated}
  If $t\in \tGa_{iso}$, then $\widetilde{L}(t)$ is countably infinite.
\end{proposition}

\begin{proof}
Let $t\in \tGa_{iso}$.
By Lemma \ref{lem:isolated-point-of-V} (i) there exists $k\in\N$ such that $D^k(t)=1-t$, so $b(t)$ is periodic with period $2k$, and $\sigma^k(b(t))=\overline{b(t)}$.
Let $x\in \wL(t)$, and assume without loss of generality that
\begin{equation} \label{eq:liminf-limsup}
\liminf D^n(x)=t \qquad\mbox{and} \qquad \limsup D^n(x)\leq 1-t.
\end{equation}
Then $b(x)$ contains arbitrarily long prefixes of $b(t)$. We will show that $b(x)$ must in fact end in $b(t)$; that is, there is some $j\in\N$ such that $\sigma^j(b(x))=b(t)$. It then follows that $\wL(t)$ is at most countable.

Suppose no such $j$ exists. Then for each $n\in\N$ there is an $m_n\in\N$ such that $\sigma^{n}(b(x))$ begins with $b_1(t)\dots b_{m_n-1}(t)$, but $b_{n+m_n}(x)\neq b_{m_n}(t)$.
Set $\tilde{m}_n:=m_n\!\mod 2k$ if this value is nonzero, and $\tilde{m}_n:=2k$ otherwise. If $b_{n+m_n}(x)<b_{m_n}(t)$ for infinitely many $n$, then for each such $n$ we have
\[
b_{n+m_n-\tilde{m}_n+1}(x)\dots b_{n+m_n}(x)\prec b_{m_n-\tilde{m}_n+1}(t)\dots b_{m_n}(t)
=b_1(t)\dots b_{\tilde{m}_n}(t)
\]
by the periodicity of $b(t)$. Since $\tilde{m}_n$ is uniformly bounded by $2k$, this implies $\liminf D^n(x)<t$, contradicting \eqref{eq:liminf-limsup}.

Otherwise, $b_{n+m_n}(x)>b_{m_n}(t)$ for infinitely many $n$. Here we consider two cases. If $\tilde{m}_n\leq k$, then
\[
b_{n+m_n-\tilde{m}_n-k+1}(x)\dots b_{n+m_n}(x)\succ b_{m_n-\tilde{m}_n-k+1}(t)\dots b_{m_n}(t)
=\overline{b_1(t)\dots b_{\tilde{m}_n+k}(t)}.
\]
On the other hand, if $k<\tilde{m}_n\leq 2k$, then
\[
b_{n+m_n-\tilde{m}_n+k+1}(x)\dots b_{n+m_n}(x)\succ b_{m_n-\tilde{m}_n+k+1}(t)\dots b_{m_n}(t)
=\overline{b_1(t)\dots b_{\tilde{m}_n-k}(t)}.
\]
Since one of these cases must hold for infinitely many $n$, it follows that $\limsup D^n(x)>1-t$, again contradicting \eqref{eq:liminf-limsup}. Hence $b(x)$ must end in $b(t)$, and $\wL(t)$ is at most countable.

On the other hand, for {\em any} $t\in\tGa$ the level set $\wL(t)$ is at least countably infinite, because $\wL(t)\ne\emptyset$ by the definition of $\tGa$, and if $x\in\wL(t)$ then $D^{-n}(\{x\})\subseteq \wL(t)$ for every $n\in\N$.
\end{proof}

Next, recall from (\ref{eq:bifurcation-dyadic}) that $\wE$ is the bifurcation set of the map $\tilde{\delta}: t\mapsto\dim_H \wS(t)$, and from (\ref{eq:Eh}) that $E^h$ is the bifurcation set of the map $\theta\mapsto\dim_H K(\theta)$. By Proposition \ref{prop:S_tau-equal-dimension} and \eqref{eq:A-B-relation} it follows that $\widetilde{E}=1-2E^h$.

\begin{lemma} \label{lem:Lambda_N}
For $N\geq 3$, let
\[
\La_N:=\left\{(d_i)_{i\geq 0}: d_0\ldots d_{2N}=01^{2N-1}0\ \textrm{and}\ 0^N\prec d_{kN+1}\ldots d_{kN+N}\prec 1^N~\forall k\ge 2\right\}.
\]
Then $\pi_2(\Lambda_N)\subseteq E^h$.
\end{lemma}

\begin{proof}
Take $(d_i)_{i\geq 0}\in\La_N$. Note that $(d_i)$ does not contain a block of $2N-1$ consecutive 0's or 1's after digit $d_{2N}$. Suppose, by way of contradiction, that there is a word $\ss=s_0\dots s_{m-1}$ with $s_0=0$ and $s_1=1$ such that
\begin{equation} \label{eq:in-this-plateau}
\ss^\f\prec (d_i) \prec \ss\overline{\ss}^\f.
\end{equation}
If $d_0\dots d_m=\ss 0$, then $m\geq 2N$ and the word $\ss$ begins with $01^{2N-1}$, so \eqref{eq:in-this-plateau} implies
\[
d_{m+1}\dots d_{m+2N-1}\lge s_1\dots s_{2N-1}=1^{2N-1},
\]
contradicting that $(d_i)\in\La_N$.

Thus, $d_0\dots d_m=\ss 1$. This implies $m\neq 2N$, so either $m<2N$ or $m>2N$. In the first case, we have necessarily $\ss=01^{m-1}$ and so $\ss\overline{\ss}^\f=01^{m-1}(10^{m-1})^\f$. By the assumption \eqref{eq:in-this-plateau}, this is only possible if $m=2N-1$. But then $\ss\overline{\ss}^\f=01^{2N-1}(0^{2N-2}1)^\f$ and $(d_i)$ begins with $\ss 1=01^{2N-1}$, i.e. $d_0\dots d_{2N-1}=01^{2N-1}$; so $d_{2N+1}\dots d_{3N}=0^N$ since $N\geq 3$ implies $2N-3\geq N$. This again contradicts that $(d_i)\in\La_N$.

If $m>2N$, then $s_1\dots s_{2N-1}=d_1\dots d_{2N-1}=1^{2N-1}$, so $\overline{\ss}^\f$ begins with $10^{2N-1}$, and \eqref{eq:in-this-plateau} implies $d_{m+1}\dots d_{m+2N-1}=0^{2N-1}$, once again contradicting that $(d_i)\in\La_N$.

Therefore, $\pi_2((d_i))$ does not lie in the interior of any plateau $(\theta_L,\theta_R)$ of $h(\theta)$. This means $\pi_2((d_i))\in E^h$, in view of \eqref{eq:Eh}.
\end{proof}

\begin{proposition} \label{prop:dimension-of-E}
The bifurcation set $\wE$ is a Cantor set, and $\dim_H \wE=\dim_H E^h=1$.
\end{proposition}

\begin{proof}
We first observe, using Proposition \ref{prop:S_tau-equal-dimension} and \eqref{eq:A-B-relation} that $\widetilde{E}=1-2E^h$, and since $E^h$ is a Cantor set, $\widetilde{E}$ is a Cantor set as well. 

Next, we show that $\dim_H \wE=1$. Define $\La_N$ as in Lemma \ref{lem:Lambda_N}. Since
\[
\dim_H\La_N=\frac{\log(2^N-2)}{N\log 2}\to 1 \qquad\mbox{as $N\to\infty$}
\]
and by Lemma \ref{lem:bi-Lip} the restriction of $\pi_2$ to $\La_N$ is bi-Lipschitz, it follows from Lemma \ref{lem:Lambda_N} that $\dim_H E^h=1$, and then also $\dim_H \widetilde{E}=\dim_H(1-2E^h)=1$. 
\end{proof}

For the remainder of this section, we introduce the symbolic sets
$$\S(t):=\pi_2^{-1}(\wB(t)), \qquad t\in[0,1].$$
Note that if $t\in\TT'$ we have
\[
\S(t)=\{(d_i)\in\Omega: b(t)\lle\sigma^n((d_i))\lle \overline{b(t)}\ \forall n\geq 0\}.
\]

\begin{lemma} \label{cor:bi-Lipschitz-in-B}
For each $0<t<1$ and each subset $F\subseteq \S(t)$,
\[
\dim_H \pi_2(F)=\dim_H F.
\]
\end{lemma}

\begin{proof}
Note that for $t\in(1/2,1)$ we have $\S(t)=\emptyset$. For $0<t\le 1/2$, the result follows from Lemma \ref{lem:bi-Lip}, since there exists $k\in\N$ such that $\S(t)\subseteq \mathbf{X}_k$.
\end{proof}

Observe that in particular, Lemma \ref{cor:bi-Lipschitz-in-B} and Proposition \ref{prop:S_tau-equal-dimension} imply
\begin{equation} \label{eq:S-B-dimension-equality}
\dim_H \S(t)=\dim_H \wB(t)=\dim_H \wS(t).
\end{equation}

\begin{proposition} \label{prop:Level-set-bifurcation-set}
For any $t\in \wE$ we have $\dim_H \wL(t)=\dim_H \wS(t)$.
\end{proposition}

\begin{proof}
Take $t\in \wE$. Since $\wL(t)\subseteq \wS(t)$, it is clear that $\dim_H \wL(t)\le\dim_H \wS(t)$.	To prove the reverse inequality, we initially fix $t'>t$ and $\eta>0$, and we construct a subset of $\wL(t)$ whose Hausdorff dimension is at least $\dim_H \wS(t')-\eta$. By the continuity of $\dim_H \wS(t)$ it then follows that $\dim_H \wL(t)\geq \dim_H \wS(t)$.

Set $\theta:=(1-t)/2$, so $\theta\in E^h$. We consider two cases: (I) $\theta\in E^h_L$; and (II) $\theta\in E^h\setminus E^h_L$.

Case I. $\theta\in E^h_L$. We set $(\alpha_i):=b(2\theta)$. As in the proof of Theorem \ref{thm:R-and-K}, there is an increasing sequence $(m_j)$ of integers such that $\alpha_1\dots\alpha_{m_j}^-$ is admissible for each $j$. (See Definition \ref{def:admissible}.) We may choose $m_1$ so that in addition,
\begin{equation} \label{eq:start-this-way}
\alpha_1\dots\alpha_{m_1-1}0^\f\succ b(2\theta'),
\end{equation}
where $\theta':=(1-t')/2$. The existence of $m_1$ in (\ref{eq:start-this-way}) follows since $\theta>\theta'$, and thus $(\al_i)=b(2\theta)\succ b(2\theta')$.
Now for each $j$, we build a {\em reset block} $R_j$ as follows.
Set $i_0=m_j$, and construct, as in the proof of Theorem \ref{thm:R-and-K}, a strictly decreasing sequence $(i_\nu: \nu=0,1,\dots,N)$ of positive integers and a sequence of exponents $(k_\nu: \nu=0,1,\dots,N-1)$ such that $\al_1\dots\al_{i_\nu}^-$ is admissible and
\begin{equation} \label{eq:big-enough-power}
\alpha_1\dots\alpha_{i_\nu(k_\nu+1)-1}\succ \alpha_1\dots\alpha_{i_\nu-1}(1\overline{\alpha_1\dots\alpha_{i_\nu-1}})^{k_\nu}, \qquad \nu=0,1,\dots,N-1.
\end{equation}
Let $\ss_\nu:=0\al_1\dots\al_{i_\nu-1}$ for $\nu=0,1,\dots,N-1$, so $\ss_0=0\alpha_1\dots\alpha_{m_j-1}$.
We set $W_\nu:=\ss_\nu^{k_\nu}$ for $\nu=0,1,\dots,N-1$, and $R_j:=W_0W_1\dots W_{N-1}0$. 
Note that the length of $R_j$ depends only on $m_j$.

Let $\mathcal N=(N_j)_{j\in\N}$ be an increasing sequence of positive integers, which we assume grows much faster than the sequence $(m_j)$. Construct sequences $(d_i)\in\Omega$ as follows: $(d_i)$ is an infinite concatenation of blocks $A_1 B_1 A_2 B_2\dots$, where $A_1$ is any word of length $N_1$ allowable in $\S(t')$; and inductively, for $j\geq 1$,
\begin{itemize}
\item let $C_j$ be the longest suffix of $A_j$ that is a prefix of either $b(\theta)$ or $\overline{b(\theta)}$. Now choose $B_j$ so that $C_j B_j=R_j$ if $C_j$ is a prefix of $b(\theta)$, or $C_j B_j=\overline{R_j}$ otherwise. This is possible since $C_j$ has length at most $m_1-1$ by \eqref{eq:start-this-way} and the fact that $A_j$ comes from $\S(t')$.
\item $A_{j+1}$ is any word of length $N_{j+1}$ allowable in $\S(t')$ beginning with $0$ if the last digit of $B_j$ is $1$, or with $1$ if the last digit of $B_j$ is $0$.
\end{itemize}
Let $F_{\mathcal N}$ denote the set of all such sequences $(d_i)=A_1 B_1 A_2 B_2\dots$. We argue that $\pi_2(F_{\mathcal N})\subseteq \wL(t)$. Take $(d_i)\in F_{\mathcal N}$ and let $x=\pi_2((d_i))$.
From the construction of the reset blocks $R_k$ it is not difficult to see that
\[
b(t)=\overline{(\alpha_i)}\prec\sigma^n((d_i))\prec (\alpha_i)=\overline{b(t)} \qquad \forall n\geq 0.
\]
Thus, $\tilde{\tau}(x)\geq t$. On the other hand, since $(d_i)$ contains arbitrarily long prefixes of $b(t)$ or $\overline{b(t)}$, we have either $\liminf D^n(x)\leq t$ or $\liminf D^n(1-x)\leq t$; in other words, $\tilde{\tau}(x)\leq t$. Hence, $\tilde{\tau}(x)=t$, and $x\in \wL(t)$.

Since we can let the ``free" blocks $A_j$, which have length $N_j$, grow much faster than the ``forced" blocks $B_j$ (which have length at most $m_j+|R_j|$), it is intuitively clear that, by choosing $N_j$ large enough, we can ensure
\begin{equation} \label{eq:same-dimension}
\dim_H F_{\mathcal N}\geq \dim_H \S(t')-\eta.
\end{equation}
(A fully rigorous proof of this fact could be modeled on the proof of \cite[Theorem 5]{Allaart-Kong-2018}.)
The sets $F_{\mathcal N}$ are contained in $\S(\hat{t})$ for any $\hat{t}<t$, {so by Lemma \ref{cor:bi-Lipschitz-in-B} we also have}
\[
\dim_H \wL(t)\geq \dim_H \pi_2(F_{\mathcal N})\geq \dim_H \wS(t')-\eta.
\]

\medskip
Case II. $\theta\in E^h\backslash E^h_L$. The construction is slightly different in this case. Note by \eqref{eq:Eh} and \eqref{eq:EhL}, that
\[
b(\theta)=\ss\overline{\ss}^\f
\]
for some word $\ss=s_0\dots s_{m-1}$ such that $s_0=0$ and $\sa:=s_1\dots s_{m-1}0$ is admissible.

Write $(\alpha_i):=b(2\theta)=\sigma(b(\theta))$. Note that $\alpha_1\dots\alpha_m^-=\sa$. Now construct a strictly decreasing sequence $(i_\nu: \nu=0,1,\dots,N)$ of positive integers as in the proof of Theorem \ref{thm:R-and-K}, starting with $i_0=m$. Although here $\theta\not\in E^h_L$, there still exist exponents $k_\nu$ for $\nu=1,\dots,N-1$ satisfying \eqref{eq:big-enough-power}, as we now explain.

Since $\theta\in E^h\backslash E^h_L$, $\theta=\theta_R(\ss)$ is the right endpoint of a maximal tuning window $[\theta_L(\ss),\theta_R(\ss)]$, and $\theta_L(\ss)\in E^h_L$ by \eqref{eq:EhL}. Set $(\alpha_i'):=b(2\theta_L(\ss))$. Then $(\alpha_i')=\sa^\f=(\alpha_1\dots\alpha_m^-)^\f$, and since $i_\nu<m$ for $\nu\geq 1$, it follows from the proof of Theorem \ref{thm:R-and-K} that there exist positive integers $k_\nu, \nu=1,2,\ldots, N-1$ such that
\begin{align*}
\al_1'\ldots \al_{i_\nu(k_\nu+1)-1}'&\succ \al_1'\ldots\al_{i_\nu-1}'\big(1\overline{\al_1'\ldots\al_{i_{\nu}-1}'}\big)^{k_\nu}\\
&=\al_1\ldots\al_{i_\nu-1}(1\overline{\al_1\ldots\al_{i_{\nu}-1}})^{k_\nu}.
\end{align*}
Since $(\alpha_i')=b(2\theta_L)\prec b(2\theta)=(\al_i)$, this proves (\ref{eq:big-enough-power}).

Now put $\ss_\nu:=0\al_1\dots\al_{i_\nu-1}$ and $W_\nu:=\ss_\nu^{k_\nu}$ for $\nu=1,2,\dots,N-1$, and let $R:=W_1\dots W_{N-1}0$. (If $N=1$, let $R$ be the ``word" $0$.)

Let $\mathcal N=(N_j)_{j\in\N}$ be an increasing sequence of positive integers, which we assume grows at least exponentially fast. As in Case I we construct sequences $(d_i)\in\Omega$ which are infinite concatenations of blocks $A_1 B_1 A_2 B_2\dots$, where $A_1$ is any word of length $N_1$ allowable in $\S(t')$; and inductively, for $j\geq 1$,
\begin{itemize}
\item let $C_j$ be the longest suffix of $A_j$ that is a prefix of either $b(\theta)$ or $\overline{b(\theta)}$. Then $B_j$ is the shortest word such that {$C_j B_j=\ss\overline{\ss}^{\ell_j}\overline{R}$ or $\overline{\ss}\ss^{\ell_j}R$ for some $\ell_j\geq j$}.
\item $A_{j+1}$ is chosen exactly as in Case I.
\end{itemize}
We note that $\ell_j>j$ only if $C_j$ already includes the word $\ss\overline{\ss}^j$ or its reflection. In this case $B_j$ simply ``finishes out" the last (possibly incomplete) copy of $\ss$ or $\overline{\ss}$ in $C_j$ before appending the reset block $R$ or $\overline{R}$. Thus, $|B_j|\leq (j+1)|\ss|+|R|$.

Let $F_{\mathcal N}$ denote the set of all such sequences $(d_i)=A_1 B_1 A_2 B_2\dots$. As in Case I, we can let the sequence $(N_j)$ grow fast enough so that \eqref{eq:same-dimension} holds, since the ``free" blocks $A_j$ from $\S(t')$, which have length $N_j$, grow much faster than the ``forced" blocks $B_j$, whose length is at most $(j+1)|\ss|+|R|$.

Now let $x=\pi_2((d_i))\in \pi_2(F_{\mathcal N})$. Then, due to the blocks $B_j$, $(d_i)$ contains either the word $\ss\overline{\ss}^{\ell}$ or the word $\overline{\ss}\ss^{\ell}$ for infinitely many $\ell$, and since $\sigma(\ss\overline{\ss}^{\ell})\to b(2\theta)=\overline{b(t)}$, it follows that either $\liminf D^n(x)\leq t$ or $\liminf D^n(1-x)\leq t$; in other words, $\tilde{\tau}(x)\leq t$.

On the other hand, although $(d_i)\not\in\S(t)$, the only values of $n$ for which $\sigma^n((d_i))\prec b(t)$ or $\sigma^n((d_i))\succ \overline{b(t)}$ are those for which $\sigma^{n-1}((d_i))$ begins with $\ss\overline{\ss}^{\ell_j}\overline{R}$ or $\overline{\ss}\ss^{\ell_j}R$, as can be seen from the admissibility of the words $\al_1\dots\al_{i_\nu}^-$, $\nu=1,\dots,N-1$. Since $\ell_j\geq j$, the words $\ss\overline{\ss}^{\ell_j}\overline{R}$ and $\overline{\ss}\ss^{\ell_j}R$ converge to $b(\theta)$ and $\overline{b(\theta)}$, respectively in the order topology as $j\to\infty$, and since $\sigma(b(\theta))=\overline{b(t)}$, this implies $\tilde{\tau}(x)\geq t$.
Hence, $\tilde{\tau}(x)=t$ and $x\in \wL(t)$. We conclude that $\pi_2(F_{\mathcal N})\subseteq \wL(t)$, and we have our desired subset.

By Cases I and II, we conclude that $\dim_H \wL(t)=\dim_H \wS(t)$ for all $t\in \wE$.
\end{proof}

\begin{remark}
Intuitively, the idea of the construction in the above proof is as follows. (We will focus on Case I, but the idea is similar for Case II.) We wish to construct points in $x\in\wL(t)$ by alternating very long finite words from $\S(t')$ (to get a dimension close to $\dim_H \wS(t')$) with ever longer initial segments of $b(t)$ or $\overline{b(t)}$ (to force $\tilde{\tau}(x)\leq t$). However, following a prefix of $b(t)$ or $\overline{b(t)}$ by an arbitrary word from $\S(t')$ might violate the conditions of membership of $\wS(t)$, hence the reset block $R_j$ is needed to act as a bridge between the $j$th prefix of $b(t)$ or $\overline{b(t)}$ and the $(j+1)$st word from $\S(t')$. In Case II the reset blocks themselves violate membership of $\wS(t)$, but the ``overshoot" becomes arbitrarily small in the limit.
\end{remark}

\subsection{The case $t\not\in \wE$: renormalization}

Having computed $\dim_H \wL(t)$ for $t\in \wE$, we now turn to the computation of this dimension for arbitrary $t\in\tGa\backslash\tGa_{iso}$. Our goal is to prove Proposition \ref{prop:tilde-level-set} below. We recall the tuning map $\Psi_I$ from Subsection \ref{subsec:Psi}.

\begin{proposition} \label{prop:tilde-level-set}
  Let $t\in\tGa\backslash \wE$, and write $t=1-2\theta$ with $\theta\in \cR\backslash E^h$. Then
  \begin{equation} \label{eq:renormalized-level-set-dimension}
  \dim_H \wL(t)=\begin{cases}
    0 & \mbox{if $\theta\in\cR_{iso}\cup\CC^h$}\\
    \frac{1}{|\ss|}\dim_H K(\hat \theta) & \textrm{otherwise},
  \end{cases}
  \end{equation}
where in the second case, $I=I(\ss)=[\theta_L,\theta_R]$ is the smallest tuning window such that $\theta\in(\theta_L,\theta_R)$, and $\hat{\theta}:=\Psi_I^{-1}(\theta)$.
\end{proposition}

The proof relies on the following technical lemma.

\begin{lemma} \label{lem:level-set-inclusion}
Let $\theta\in\cR\cap I$ for some tuning window $I=I(\ss)$, and let $\hat{\theta}:=\Psi_I^{-1}(\theta)$. Then
\begin{equation} \label{eq:level-set-inclusion}
\wL(1-2\theta)\supseteq \Psi_I\Big(\wL(1-2\hat{\theta})\Big),
\end{equation}
and consequently,
\begin{equation} \label{eq:level-set-dimension-inequality}
\dim_H \wL(1-2\theta)\geq \frac{1}{|\ss|}\dim_H \wL(1-2\hat{\theta}).
\end{equation}
\end{lemma}

\begin{proof}
Write $\hat{\theta}=.\hat{\theta}_1\hat{\theta}_2\dots$, and note that $\hat{\theta}_1=0$. Take $\hat{x}=.\hat{x}_1\hat{x}_2\dots\in \wL(1-2\hat{\theta})$. Then $\tilde{\tau}(\hat{x})=1-2\hat{\theta}$, and we may assume without loss of generality that $\liminf D^n(1-\hat{x})=1-2\hat{\theta}$, or equivalently, $\limsup D^n(\hat{x})=2\hat{\theta}$. Set $x:=\Psi_I(\hat{x})$ and $t:=1-2\theta$. We must show $\tilde{\tau}(x)=t$.

We first show that $\tilde{\tau}(x)\leq t$. Since $\limsup D^n(\hat{x})=2\hat{\theta}$, there is for each $k\in\N$ an index $n_k$ such that $\hat{x}_{n_k+1}\dots \hat{x}_{n_k+k-1}=\hat{\theta}_2\dots \hat{\theta}_k$. We may assume that $\hat{x}_{n_k}=0=\hat{\theta}_1$ for all but finitely many $k$, as otherwise we would have
\[
\limsup D^n(\hat{x})\geq .1\hat{\theta}_2\hat{\theta}_3\dots=\frac12(1+.\hat{\theta}_2\hat{\theta}_3\dots)=\frac12(1+2\theta)>2\theta
\]
(since $2\theta<1$), a contradiction. Thus, for all large enough $k$, $\hat{x}_{n_k}\dots \hat{x}_{n_k+k-1}=\hat{\theta}_1\dots\hat{\theta}_k$. Hence $.\Sigma_{\hat{x}_{n_k}}\dots\Sigma_{\hat{x}_{n_k+k-1}}=.\Sigma_{\hat{\theta}_1}\dots\Sigma_{\hat{\theta}_k}$, which converges to $\Psi_I(\hat{\theta})=\theta$ as $k\to\infty$. This yields $D^{|\ss|(n_k-1)+1}(x)\to 2\theta$ (since the first digit of $\Sigma_{\hat{x}_{n_k}}=\Sigma_0=\ss$ is a $0$), and so
\[
\tilde{\tau}(x)\leq \lim_{k\to\f} D^{|\ss|(n_k-1)+1}(1-x)=1-2\theta=t.
\]

The proof of the reverse inequality is more involved. Write $\theta=.\theta_1\theta_2\dots$ and set $x=.x_1x_2x_3\dots$. Suppose, by way of contradiction, that $\tilde{\tau}(x)<t$, and assume without loss of generality that $\liminf D^n(1-x)<t=1-2\theta$. Then $\limsup D^n(x)>2\theta$, so we can find a sequence $(n_k)$ of indices such that $\lim_{k\to\f} D^{n_k}(x)>2\theta$ and $x_{n_k-1}=0$ for every $k$. Then $\lim_{k\to\f} D^{n_k-1}(x)\in(\theta,1/2)$, so there is some word $w_1\dots w_l$ that occurs infinitely often in the sequence $(x_i)$ such that $w_1\dots w_l\succ \theta_1\dots \theta_l$ and $w_1=0$. We may assume without loss of generality that $l>m:=|\ss|$, so
\begin{equation} \label{eq:w-word-and-t-word}
w_1\dots w_{m+1}\succeq \theta_1\dots \theta_{m+1}.
\end{equation}

Now let $j$ be any integer such that $w_1\dots w_l=x_{j+1}\dots x_{j+l}$.
We claim that $j$ must be a multiple of $m$. 
This may be seen as follows. Since $\ss$ is a tuning base, the word $\sa:=s_1\dots s_{m-1}0$ satisfies \eqref{eq:admissible-definition}. In other words,
\begin{equation} \label{eq:admissible-s}
\overline{s_1\dots s_{m-i}}\lle s_{i+1}\dots s_{m-1}0\prec s_1\dots s_{m-i} \qquad\forall\, 1\leq i<m.
\end{equation}
Recall also that $\theta_1\dots \theta_{m}=\ss$, since $\hat{\theta}_1=0$. Since $x=\Psi_I(\hat{x})$, we may write $x=.\b_1\b_2\dots$, where each $\b_k\in\{\ss,\overline{\ss}\}$.
Suppose now that $j$ is not a multiple of $m$, and write $j=(k-1)m+i$, where $k\in\N$ and $1\leq i<m$. Then $w_1\dots w_{m-i}$ is a suffix of $\b_k$, and $w_{m-i+1}\dots w_{m+1}$ is a prefix of $\b_{k+1}$. We now have four cases:

(i) $\b_k=\b_{k+1}=\ss$. Then $w_1\dots w_{m-i}=s_i\dots s_{m-1}$ and $w_{m-i+1}\dots w_{m+1}=s_0\dots s_i$. Since $w_1=0$, this implies
\[
w_1\dots w_{m-i+1}=0s_{i+1}\dots s_{m-1}0\prec 0s_1\dots s_{m-i}=s_0\dots s_{m-i}=\theta_1\dots\theta_{m-i+1}
\]
by \eqref{eq:admissible-s}, contradicting \eqref{eq:w-word-and-t-word}.

(ii) $\b_k=\ss$, $\b_{k+1}=\overline{\ss}$. Then $w_1\dots w_{m-i}=s_i\dots s_{m-1}$ and $w_{m-i+1}\dots w_{m+1}=\overline{s_0\dots s_i}$. So
\[
w_1\dots w_{m-i+1}=0s_{i+1}\dots s_{m-1}1\lle s_0\dots s_{m-i}
\]
and
\[
w_{m-i+2}\dots w_{m+1}=\overline{s_1\dots s_i}\lle s_{m-i+1}\dots s_{m-1}0.
\]
The last inequality is in fact strict, because $s_i=w_1=0$ and so $w_{m+1}=\overline{s_i}=1$. Hence, $w_1\dots w_{m+1}\prec s_0\dots s_{m-1}0\leq \theta_1\dots\theta_{m+1}$, again contradicting \eqref{eq:w-word-and-t-word}.

(iii) $\b_k=\overline{\ss}$, $\b_{k+1}=\ss$. Then $w_1\dots w_{m-i}=\overline{s_i\dots s_{m-1}}$ and $w_{m-i+1}\dots w_{m+1}=s_0\dots s_i$, so
{\begin{align*}
w_1\dots w_{m-i+1}&=0\,\overline{s_{i+1}\dots s_{m-1}}\,0\prec 0\,\overline{s_{i+1}\dots s_{m-1}0}\\
&\lle 0 s_1\dots s_{m-i}=s_0\dots s_{m-i}=\theta_1\dots\theta_{m-i+1},
\end{align*}
}
again contradicting \eqref{eq:w-word-and-t-word}.

(iv) $\b_k=\b_{k+1}=\overline{\ss}$. Then $w_1\dots w_{m-i}=\overline{s_i\dots s_{m-1}}$ and $w_{m-i+1}\dots w_{m+1}=\overline{s_0\dots s_i}$. {Since $w_1=0$, this implies $s_i=1$. So
\begin{align*}
w_1\dots w_{m+1}&=0\,\overline{s_{i+1}\dots s_{m-1}s_0\dots s_i}=0\,\overline{a_{i+1}\dots a_m a_1\dots a_i}\\
&\prec 0a_1\dots a_m=s_0\dots s_{m-1}0\lle \theta_1\dots\theta_{m+1},
\end{align*}
where the second equality uses that $s_0=a_m=0$, and the first inequality follows by Lemma \ref{lem:admissible-word-property}. This once again contradicts \eqref{eq:w-word-and-t-word}.}

Now that we have established that $j$ must be a multiple of $m$, we see that there is a sequence $(\ell_k)$ of indices such that
\[
\frac12>\lim_{k\to\f}\Psi_I\big(D^{\ell_k}(\hat{x})\big)=\lim_{k\to\f} D^{m\ell_k}(x)>\theta=\Psi_I(\hat{\theta}),
\]
and so
\[
\frac12>\lim_{k\to\f}D^{\ell_k}(\hat{x})>\hat{\theta},
\]
since the binary expansion of $D^{m\ell_k}(x)$ begins with $\ss$, so the binary expansion of $D^{\ell_k}(\hat{x})$ begins with $0$. As a result,
\[
\tilde{\tau}(\hat{x})\leq\lim_{k\to\f}D^{\ell_k+1}(1-\hat{x})=1-\lim_{k\to\f}D^{\ell_k+1}(\hat{x})<1-2\hat{\theta},
\]
contradicting our initial assumption that $\hat{x}\in \wL(1-2\hat{\theta})$. Hence, $\tilde{\tau}(x)=t$. This completes the proof of \eqref{eq:level-set-inclusion}.

To derive \eqref{eq:level-set-dimension-inequality}, we define for $t\in\tGa$ the subsets
\[
\wL_k(t):=\wL(t)\cap \pi_2(\mathbf{X}_k), \qquad k\in\N,
\]
where $\mathbf{X}_k$ was defined just above Lemma \ref{lem:bi-Lip}.
If $t>0$, then there is an integer $k=k(t)$ such that for each $x\in \wL(t)$, the binary expansion of $x$ does not contain $k$ consecutive 0's or 1's after a certain point. Hence
\[
\wL(t)=\bigcup_{n=0}^\f D^{-n}(\wL_k(t)),
\]
which implies $\dim_H \wL(t)=\dim_H \wL_k(t)$. Now let $k:=k(1-2\hat{\theta})$. By Lemma \ref{lem:Psi-Holder} it follows that $\Psi_I$ is bi-H\"older continuous with exponent $|\ss|$ when restricted to $\wL_k(1-2\hat{\theta})$. Clearly from \eqref{eq:level-set-inclusion},
\[
\wL(1-2\theta)\supseteq \Psi_I\Big(\wL_k(1-2\hat{\theta})\Big),
\]
and hence
\[
\dim_H \wL(1-2\theta)\geq \frac{1}{|\ss|}\dim_H \wL_k(1-2\hat{\theta})=\frac{1}{|\ss|}\dim_H \wL(1-2\hat{\theta}),
\]
establishing \eqref{eq:level-set-dimension-inequality}.
\end{proof}

\begin{proof}[Proof of Proposition \ref{prop:tilde-level-set}]
If $\theta\in\cR_{iso}$, then $t=1-2\theta\in \tGa_{iso}$ and \eqref{eq:renormalized-level-set-dimension} follows directly from Proposition \ref{prop:level-set-isolated}.

{Suppose there is a smallest tuning window $I=I(\ss)=[\theta_L, \theta_R]$ such that $\theta\in(\theta_L,\theta_R)$. Then $\theta\in E^h(I)$ by the proof of Lemma \ref{lem:R-decomposition}, so $\hat\theta=\Psi_I^{-1}(\theta)\in E^h$ by \eqref{eq:bifurcation-tuning}. Thus, Lemma \ref{lem:level-set-inclusion}, Proposition \ref{prop:Level-set-bifurcation-set}, Proposition \ref{prop:S_tau-equal-dimension} and \eqref{eq:A-B-relation} imply}
\begin{equation} \label{eq:level-set-lower}
\dim_H \wL(t)\geq \frac{1}{|\ss|}\dim_H \wL(1-2\hat{\theta})=\frac{1}{|\ss|}\dim_H \wS(1-2\hat{\theta})=\frac{1}{|\ss|}\dim_H K(\hat{\theta}).
\end{equation}

For the upper bound, we note that $\wL(t)=\wL_1(t)\cup [1-\wL_1(t)]$, where
\begin{align*}
\wL_1(t):&=\left\{x\in \wL(t): \liminf_{n\to\f} D^n(1-x)=t\right\}\\
&=\left\{x\in \wL(1-2\theta): \limsup_{n\to\f} D^n(x)=2\theta\right\}.
\end{align*}
So it suffices to establish the upper bound in \eqref{eq:renormalized-level-set-dimension} for $\wL_1(t)$. Write $\theta=.\theta_1\theta_2\dots$, and take $x=.x_1x_2\dots\in \wL_1(t)$. Then $\limsup D^n(x)=2\theta$, which implies as in the proof of Lemma \ref{lem:level-set-inclusion} that for each $k\in\N$ there is an index $n_k$ such $x_{n_k}\dots x_{n_k+k-1}=\theta_1\dots\theta_k$, so that $|D^{n_k-1}(x)-\theta|<2^{-k}$. Hence, for large enough $k$, $D^{n_k-1}(x)\in I$.
Furthermore, for each $\theta'>\theta$ there exists $N\in\N$ such that $D^n(x)\in K(\theta')$ for any $n\ge N$. Hence, given $\theta'\in(\theta,\theta_R)$, there exists $n\in\N$ such that $D^n(x)\in K_I(\theta')$. We conclude from this discussion that
\begin{equation*}
  \wL_1(t)\subseteq\bigcup_{n=0}^\f D^{-n}\big(K_I(\theta')\big)\quad\forall \,\theta'\in(\theta,\theta_R),
\end{equation*}
so by Lemma \ref{lem:K_I-continuity} and \eqref{eq:K-theta-tuning-dimension},
\begin{equation*}
\dim_H \wL(t)=\dim_H \wL_1(t)\le \lim_{\theta'\searrow \theta} \dim_H K_I(\theta')
=\dim_H K_I(\theta)=\frac{1}{|\ss|}\dim_H K(\hat \theta).
\end{equation*}
This, together with (\ref{eq:level-set-lower}), proves (\ref{eq:renormalized-level-set-dimension}) for $\theta\in E^h(I)$.

By \eqref{eq:R-decomposition} it remains to prove (\ref{eq:renormalized-level-set-dimension}) for $\theta\in\CC^h$. Such a $\theta$ lies in infinitely many tuning windows, so $\dim_H \wL(t)\le \frac{1}{|\ss|}\dim_H K(\hat \theta)$ for infinitely many words $\ss$ and corresponding $\hat \theta\in E^h$. This implies that $\dim_H \wL(t)=0$, completing the proof.
\end{proof}



\begin{proof}[Proof of Theorem \ref{main2:level-set}]
The theorem follows from Propositions \ref{cor:R-level-1}, \ref{prop:R-local-dimension}, \ref{prop:level-set-isolated}, \ref{prop:dimension-of-E}, \ref{prop:Level-set-bifurcation-set}, and \ref{prop:tilde-level-set}, bearing in mind the relationships $\tGa=1-2(\cR\backslash\{0\})$ and $\dim_H K(\theta)=\dim_H \wB(1-2\theta)=\dim_H \wS(1-2\theta)$.
\end{proof}

\section{Proofs of the main results} \label{sec:main-proofs}

In this section we deduce Theorems \ref{main:Ga-topology}, \ref{main:Ga-dimension} and \ref{main:level-set} from their counterparts in Section \ref{sec:preliminaries}.

\begin{proof}[Proof of Theorem \ref{main:Ga-topology}]
The representation \eqref{eq:Gamma-representation} follows from \eqref{eq:Gamma-representation2} since $\Ga=\psi^{-1}(\tGa)$ and $D^n\circ \psi=\psi\circ T^n$ for all $n$.

Next, recall that the map $\psi: \Ga\to \tGa$ is a homeomorphism. Thus, by Theorem \ref{main2:Ga-topology} it follows that $\Ga$ is a closed subset of $C$, and it contains infinitely many isolated and infinitely many accumulation points. Furthermore, $0=\psi^{-1}(0)=\min\Ga$ is an accumulation point of $\Ga$, and $\rho/(1+\rho)=\pi((01)^\f)=\psi^{-1}(1/3)=\max\Ga$ is isolated in $\Ga$. Finally, statement (i) follows from Theorem \ref{main2:Ga-topology} (i).
\end{proof}

We next recall the relations \eqref{eq:psi-correspondence}. Unfortunately the map $\psi$ is not bi-H\"older continuous when restricted to $S(t)$ or even $L(t)$, so we have to be somewhat more careful.
Set
\[
S_k(t):=S(t)\cap \pi(\mathbf{X}_k), \qquad L_k(t):=L(t)\cap \pi(\mathbf{X}_k),
\]
and similarly,
\[
\wS_k(t):=\wS(t)\cap \pi_2(\mathbf{X}_k), \qquad \wL_k(t):=\wL(t)\cap \pi_2(\mathbf{X}_k),
\]
where $\mathbf{X}_k$ was defined just above Lemma \ref{lem:bi-Lip}.
Then for each $k\in\N$ and $t\in C'\backslash\{0\}$ we have
\[
S_k(t)=\psi^{-1}\big(\wS_k(\psi(t))\big), \qquad L_k(t)=\psi^{-1}\big(\wL_k(\psi(t))\big),
\]
so by Lemmas \ref{lem:bi-Holder} and \ref{lem:bi-Lip} it follows that the restriction of $\psi$ to $S_k(t)$ is bi-H\"older continuous with exponent $s(\rho)$. Hence, for each $k\in\N$ and $t\in C'\backslash\{0\}$,
\begin{equation} \label{eq:dimension-relations}
\dim_H S_k(t)=s(\rho)\dim_H \wS_k(\psi(t)), \qquad \dim_H L_k(t)=s(\rho)\dim_H \wL_k(\psi(t)).
\end{equation}

\begin{proof}[Proof of Theorem \ref{main:Ga-dimension}]
Given $t\in C'\backslash\{0\}$, we choose $k\in\N$ so that
\[
\psi(\rho t)>2^{-k}.
\]
If $x\in S(t)$, then $\tau(x)\geq t>\rho t$, so there is an integer $n_0$ such that $T^n(x)\geq\rho t$ and $T^n(1-x)\geq \rho t$ for all $n\geq n_0$. Applying $\psi$ to both sides of these inequalities, we obtain also that $D^n(\psi(x))\geq \psi(\rho t)>2^{-k}$ and $D^n(\psi(1-x))>2^{-k}$ for all $n\geq n_0$. This implies that
\[
T^{n_0}(x)\in \pi(\mathbf{X}_k) \qquad\mbox{and} \qquad D^{n_0}(\psi(x))\in \pi_2(\mathbf{X}_k),
\]
and hence,
\[
T^{n_0}(x)\in S_k(t) \qquad\mbox{and} \qquad D^{n_0}(\psi(x))\in \wS_k(\psi(t)).
\]
We conclude from this discussion that
\[
S(t)=\bigcup_{n=0}^\f T^{-n}\big(S_k(t)\big)\qquad\mbox{and}\qquad \wS(\psi(t))=\bigcup_{n=0}^\f D^{-n}\big(\wS_k(\psi(t))\big).
\]
Therefore, using \eqref{eq:dimension-relations},
\begin{equation} \label{eq:S-dim-inequality}
\dim_H S(t)=\dim_H S_k(t)=s(\rho)\dim_H \wS_k(\psi(t))=s(\rho)\dim_H \wS(\psi(t)),
\end{equation}
which combined with Theorem \ref{th:survivor set} proves (iii).

Similarly, for any $t\in C$ we can find $k\in\N$ such that $\pi^{-1}(\Ga\cap[t,1])\in \mathbf{X}_k$, so another application of Lemma \ref{lem:bi-Lip} yields that the restriction of $\psi$ to $\Ga\cap[t,1]$ is bi-H\"older continuous with exponent $s(\rho)$. Hence, for all $t\in C'\backslash\{0\}$,
\[
\dim_H\big(\Ga\cap[t,1]\big)=s(\rho)\dim_H\big(\tGa\cap[\psi(t),1]\big)=s(\rho)\dim_H \wS(\psi(t))=\dim_H S(t),
\]
and by extending to $t\in(0,1]$ (easy exercise) we obtain (ii). Finally, (i) follows from (ii) by letting $t\to 0$, using Theorem \ref{th:survivor set}, and noting that $S(0)=C$.
\end{proof}

\begin{remark} \label{rem:devil-staircase-detail}
From Theorem \ref{th:survivor set}, \eqref{eq:psi-correspondence} and \eqref{eq:S-dim-inequality} we obtain additional information about the size of $S(t)$: Let $t_F:=\psi^{-1}(\tilde{t}_F)$, and recall that $\psi^{-1}(1/3)=\rho/(1+\rho)$.
\begin{enumerate}[(i)]
\item If $\rho/(1+\rho)<t\le 1$, then $S(t)=\emptyset$.
\item If $t_F<t\le \rho/(1+\rho)$, then $S(t)$ is countably infinite.
\item If $t=t_F$, then $S(t)$ is uncountable and $\dim_H S(t)=0$.
\item If $0\le t<t_F$, then $\dim_H S(t)>0$.
\end{enumerate}
\end{remark}

\begin{proof}[Proof of Theorem \ref{main:level-set}]
From Proposition \ref{prop:level-set-isolated} and \eqref{eq:psi-correspondence} we deduce that $L(t)$ is countably infinite for $t\in \Ga_{iso}$, since $\psi:\Ga\to\tGa$ is a homeomorphism.

Next, by similar reasoning as in the proof of Theorem \ref{main:Ga-dimension}, we have
\begin{equation} \label{eq:level-set-equality}
\dim_H L(t)=s(\rho)\dim_H \wL(\psi(t)),
\end{equation}
and so Theorem \ref{main2:level-set} implies
\begin{align*}
\lim_{\ep\searrow 0}\dim_H\big(\Ga\cap[t-\ep,t+\ep]\big)&=s(\rho)\lim_{\delta\searrow 0}\dim_H\big(\tGa\cap[\psi(t)-\delta,\psi(t)+\delta]\big)\\
&=s(\rho)\dim_H \wL(\psi(t))=\dim_H L(t).
\end{align*}

Finally, $E=\psi^{-1}(\widetilde{E})$ by \eqref{eq:S-dim-inequality}, so
\[
\dim_H E=s(\rho)\dim_H \wE=s(\rho)=\dim_H \Ga,
\]
and \eqref{eq:S-dim-inequality} and \eqref{eq:level-set-equality} yield \eqref{eq:level-S}.
\end{proof}

\begin{remark} \label{rem:renormalized-level-set}
In order to obtain the Hausdorff dimension of $L(t)$ for $t\in\Ga\backslash E$, simply replace $t$ with $\psi(t)$ in Proposition \ref{prop:tilde-level-set}, and use \eqref{eq:level-set-equality}.
\end{remark}

We end the paper with a concrete example illustrating Theorems \ref{main:Ga-dimension} and \ref{main:level-set}.

\begin{example} \label{ex:longest-plateau}
  Since $E=\psi^{-1}(1-2E^h)=:\Phi(E^h)$ and $E^h$ was extensively studied in \cite{Douady-1995}, we can use $E^h$ to describe the bifurcation set $E$. Observe that the longest connected component of $[\theta_F,1]\backslash E^h$ is the interval $(\theta_L, \theta_R)$ determined by
 \[
 \theta_L=.(011)^\f\quad\textrm{and}\quad \theta_R=.011(100)^\f.
 \]
 This implies that the longest connected component of $[0, t_F]\setminus E$ is given by
 \[
 (t_L, t_R)=(\Phi(\theta_R), \Phi(\theta_L))=\big(\pi(000(110)^\f), \pi((001)^\f)\big).
 \]
 Thus, for example, if $\rho=1/3$ then the longest plateau of the map $t\mapsto \dim_H S(t)$ is $[\pi(000(110)^\f), \pi((001)^\f)]=[\frac{4}{117}, \frac{1}{13}]$ (see Figure \ref{fig:1}), and for any $t$ in this interval we have
 \begin{align*}
 \dim_H(\Ga\cap[t,1])&=\dim_H S(1/13)\\
 &=s(1/3)\dim_H\set{(d_i): (001)^\f\lle\si^n((d_i))\lle (110)^\f~\forall n\ge 0}\\
 &=\frac{\log((1 +\sqrt{5})/2)}{\log 3}.
\end{align*}
Furthermore, this is also the value of $\dim_H L(1/13)$, in view of Theorem \ref{main:level-set} (iii).
\end{example}

\section*{Acknowledgements}
{We are grateful to two anonymous referees for several helpful suggestions, and in particular to one referee who brought} the works of Douady and Tiozzo to our attention, which led to a thorough revision of the manuscript.
D.~Kong thanks De-Jun Feng for bringing the Cantor densities to his attention and pointing out their connection to open dynamical systems at {a meeting in Edinburgh in 2018.} 
P. Allaart is partially supported by Simons Foundation grant \#709869.
D.~Kong was supported by NSFC No.~11971079.


\begin{thebibliography}{10}

\bibitem{Alcaraz_Barrera_2014}
R.~Alcaraz~Barrera,
\newblock Topological and ergodic properties of symmetric sub-shifts,
\newblock {\em Discrete Contin. Dyn. Syst.}, {\bf 34} (2014), no. 11, 4459--4486.

\bibitem{AlcarazBarrera-Baker-Kong-2016}
R.~Alcaraz~Barrera, S.~Baker, and D.~Kong.
\newblock Entropy, topological transitivity, and dimensional properties of unique {$q$}-expansions.
\newblock {\em Trans. Amer. Math. Soc.}, 371(5):3209--3258, 2019.




\bibitem{Allaart-Kong-2018}
P.~Allaart and D.~Kong.
\newblock Relative bifurcation sets and the local dimension of univoque bases.
\newblock {\em Ergodic Theory Dynam. Systems}, 41(8):2241--2273, {2021.}

\bibitem{Allouche_Cosnard_1983}
J.-P. Allouche and M.~Cosnard.
\newblock It\'erations de fonctions unimodales et suites engendr\'ees par automates.
\newblock {\em C. R. Acad. Sci. Paris S\'er. I Math.}, 296(3):159--162, 1983.


\bibitem{Allouche-Cosnard-2001}
J.-P. Allouche and M.~Cosnard.
\newblock Non-integer bases, iteration of continuous real maps, and an arithmetic self-similar set.
\newblock {\em Acta Math. Hungar.}, 91(4):325--332, 2001.

\bibitem{Allouche_Shallit_1999}
J.-P. Allouche and J.~Shallit.
\newblock The ubiquitous {P}rouhet-{T}hue-{M}orse sequence.
\newblock In {\em Sequences and their applications ({S}ingapore, 1998)},
  Springer Ser. Discrete Math. Theor. Comput. Sci., pages 1--16. Springer, London, 1999.


\bibitem{BF-1992}
T.~Bedford and A.~Fisher.
\newblock Analogues of the Lebesgue density theorem of fractal sets of real and integers.
\newblock{\em Proc. London Math. Soc.} 64(3):95--124, 1992.

\bibitem{Bon-Car-Ste-Giu-2013}
C.~Bonanno, C.~Carminati, S.~Isola, and G.~Tiozzo.
\newblock Dynamics of continued fractions and kneading sequences of unimodal maps.
\newblock {\em Discrete Contin. Dyn. Syst.}, 33(4):1313--1332, 2013.

\bibitem{CT-2021}
C.~Carminati and G.~Tiozzo.
\newblock The bifurcation locus for numbers of bounded type.
\newblock{\em Ergodic Theory Dynam. Systems, doi:10.1017/etds.2021.28}, 2021.



\bibitem{Douady-1995}
A. Douady.
\newblock Topological entropy of unimodal maps: monotonicity for quadratic polynomials.
\newblock {\em Real and complex dynamical systems (Hiller?d, 1993)}, 65--87,
\newblock NATO Adv. Sci. Inst. Ser. C: Math. Phys. Sci., 464, Kluwer, Dordrecht, 1995

\bibitem{Falconer_1990}
K.~Falconer.
\newblock {\em Fractal geometry: Mathematical foundations and applications}.
\newblock John Wiley \& Sons, Ltd., Chichester, third edition, 2014.

\bibitem{Feng-Hua-Wen-2000}
D.-J. Feng, S.~Hua, and Z.-Y. Wen.
\newblock The pointwise densities of the {C}antor measure.
\newblock {\em J. Math. Anal. Appl.}, 250(2):692--705, 2000.


\bibitem{Kalle-Kong-Li-Lv-2016}
C.~Kalle, D.~Kong, W.~Li, and F.~L\"{u}.
\newblock On the bifurcation set of unique expansions.
\newblock {\em Acta Arith.}, 188(4):367--399, 2019.


\bibitem{Komornik_Loreti_2007}
V.~Komornik and P.~Loreti.
\newblock On the topological structure of univoque sets.
\newblock {\em J. Number Theory}, 122(1):157--183, 2007.


\bibitem{KLY-2021}
D.~Kong, W.~Li and Y.~Yao.
\newblock Pointwise densities of homogeneous Cantor measure and critical values.
\newblock{\em Nonlinearity}, 34: 2350--2380, 2021.

\bibitem{Lind_Marcus_1995}
D.~Lind and B.~Marcus.
\newblock {\em An introduction to symbolic dynamics and coding}.
\newblock Cambridge University Press, Cambridge, 1995.

\bibitem{Mattila_1995}
P.~Mattila.
\newblock {\em Geometry of sets and measures in {E}uclidean spaces. Fractals and rectifiability}, volume~44 of {\em Cambridge Studies in Advanced Mathematics}.
\newblock Cambridge University Press, Cambridge, 1995.

\bibitem{Raith_1989}
P.~Raith.
\newblock Hausdorff dimension for piecewise monotonic maps.
\newblock {\em Studia Math.},  94(1):17--33, 1989.

\bibitem{Tiozzo_2015}
G.~Tiozzo.
\newblock Topological entropy of quadratic polynomials and dimension of sections of the Mandelbrot set.
\newblock {\em Adv. Math.}, 273:651--715, 2015.

\bibitem{Tiozzo_2018}
G.~Tiozzo.
\newblock The local H\"older exponent for the entropy of real unimodal maps.
\newblock {\em Sci. China Math.}, 61(12):2299--2310, 2018.

\bibitem{Urbanski-1986}
M.~Urba\'{n}ski.
\newblock On {H}ausdorff dimension of invariant sets for expanding maps of a circle.
\newblock {\em Ergodic Theory Dynam. Systems}, 6(2):295--309, 1986.

\bibitem{Walters_1982}
P.~Walters.
\newblock {\em An introduction to ergodic theory}.
\newblock Springer-Verlag, New York-Berlin,  1982.

\bibitem{Wang-Wu-Xiong-2011}
J.~Wang, M.~Wu, and Y.~Xiong.
\newblock On the pointwise densities of the {C}antor measure.
\newblock {\em J. Math. Anal. Appl.}, 379(2):637--648, 2011.

\end{thebibliography}

\end{document}